\documentclass[12pt,reqno]{amsart}
\usepackage[shortlabels]{enumitem}
\usepackage{esint}
\usepackage{physics}
\usepackage{relsize}
\usepackage[backref]{hyperref}
\usepackage{mathtools}
\usepackage{amsfonts}
\usepackage[all]{xy}
\usepackage{geometry}
\usepackage{amsmath,amssymb} 
\usepackage{faktor}

\newtheorem{theorem}{Theorem}[section]
\newtheorem{lemma}[theorem]{Lemma}
\newtheorem{definition}[theorem]{Definition}

\theoremstyle{corollary}
\newtheorem{corollary}[theorem]{Corollary}
\theoremstyle{conjecture}

\newtheorem*{claim}{Claim}
\theoremstyle{assumption}

\theoremstyle{proposition}
\newtheorem{proposition}[theorem]{Proposition}
\theoremstyle{remark}
\newtheorem{remark}[theorem]{Remark}
\numberwithin{equation}{section}
\everymath{\displaystyle}

\newcommand{\olsi}[1]{\,\overline{\!{#1}}} 
\newcommand{\ols}[1]{\mskip.5\thinmuskip\overline{\mskip-.5\thinmuskip {#1} \mskip-.5\thinmuskip}\mskip.5\thinmuskip} 
\newcommand{\ca}{\mathcal C}
\newcommand{\ii}{\ensuremath{\sqrt{-1}}}
\newcommand{\pp}{\bar\partial}
\newcommand{\C}{\mathbb{C}}
\DeclareMathOperator{\vol}{Vol}
\DeclareMathOperator{\Rm}{Rm}
\DeclareMathOperator{\Ric}{Ric}

\begin{document}

\title[Polynomial convergence to tangent cones]{On polynomial convergence to tangent cones for singular K\"ahler-Einstein metrics}
\author{Junsheng Zhang}
\address{Department of Mathematics\\
 University of California\\
 Berkeley, CA, USA, 94720\\}
\curraddr{}
\email{jszhang@berkeley.edu}
\thanks{}

\maketitle
\begin{abstract}
Let $(Z,p)$ be a pointed Gromov-Hausdorff limit of non-collapsing K\"ahler-Einstein metrics with uniformly bounded Ricci curvature.
We show that the singular K\"ahler-Einstein metric on $Z$ is conical at $p$ if and only if $\ca=W$ in Donaldson-Sun's two-step degeneration theory, assuming curvature grows at most quadratically near $p$.

Let $(X,p)$ be a germ of an isolated log terminal algebraic singularity. Following Hein-Sun's approach, we show that if $\ca=W$ in the two-step stable degeneration of $(X,p)$ and $\ca$ has a smooth link, then every singular K\"ahler-Einstein metric  on $X$ with non-positive Ricci curvature and bounded potential is conical at $p$.
\end{abstract}

\section{Introduction}\label{s1}
In this paper, we study the geometry of singular K\"ahler-Einstein metrics. Such metrics arise from two main sources: Gromov-Hausdorff limits of smooth K\"ahler-Einstein metrics and pluripotential theory.
 Similar to other problems in geometric analysis, a crucial approach to understanding the structure of a singular K\"ahler-Einstein metric near its singularity involves examining its tangent cones. This process typically involves three steps:
\begin{itemize}
\item[(1).] The existence of tangent cones;
\item [(2).] The uniqueness of tangent cones;
\item [(3).] Relate the geometry of $Z$ near $p$ with the geometry on the tangent cone at $p$. 
\end{itemize} In this paper, we make progress on step (3) under certain assumptions that guarantee the existence and uniqueness of tangent cones.

Firstly we consider singular K\"ahler-Einstein metrics coming from the Gromov-Hausdorff limits. Let $(Z, p, d_Z)$ be a pointed Gromov-Hausdorff limit of complete non-collapsing K\"ahler-Einstein manifolds with uniformly bounded Ricci curvature. According to \cite{ChCo1,An,CCT,CN2012}, there is a decomposition $Z=Z^{reg}\sqcup Z^{sing}$, where the regular part $Z^{reg}$ consists of points around which a neighborhood is a smooth Riemannian manifold. This $Z^{reg}$ is an open, connected, smooth manifold, and the restriction of $d_Z$ to $Z^{reg}$ is induced by a K\"ahler-Einstein metric $(g_Z, J_Z, \omega_Z)$, which is referred to as the singular K\"ahler-Einstein metric on $Z$.   
    
   If one further assumes that the smooth K\"ahler-Einstein metrics are polarized, i.e. the K\"ahler forms are curvature forms of line bundles, and that the canonical bundle is homomorphically trivial for Ricci-flat metrics, then  Donaldson and Sun \cite{DS1,DS2} showed that $Z$ has the structure of a normal analytic space with log terminal singularities and the tangent cone $\ca$ at $p$ is unique. Furthermore, a two-step degeneration exists, proceeding from the local ring $\mathcal{O}_{Z,p}$ through an \emph{intermediate K-semistable cone} $W$ to $\ca$. This is referred to as Donaldson-Sun's 2-step degeneration theory. 
  It is an open problem whether Donaldson-Sun's 2-step degeneration theory holds without the polarization condition.

In this paper, we assume $p$ is an isolated singularity and that near $p$, the curvature grows at most quadratically. That is, there exists a constant $C>0$ such that the following holds near $p$
  \begin{equation}\label{eq--curvature decay}
  	|\Rm(\cdot)|\leq C d_Z(p,\cdot)^{-2}, 
  \end{equation}however the polarization condition is not required. The main result is as follows.
   \begin{theorem}\label{thm-main theorem}
   Suppose \eqref{eq--curvature decay} holds. Then Donaldson-Sun's 2-step degeneration theory applies. Moreover the singular K\"ahler-Einstein metric is conical at $p$ if and only $\ca$ is isomorphic to $W$ as Fano cones in Donalson-Sun's 2-step degeneration theory.	
 \end{theorem}
   
   We refer to Section \ref{sec-terminologies} for the definition of Fano cones and to Section \ref{sec--ds theory} for further details on Donaldson-Sun's 2-step degeneration theory. We will simply use $\ca=W$ to denote that $W$ is isomorphic to $\ca$ as Fano cones and in the following.
    We remark that assuming \eqref{eq--curvature decay}, Donaldson-Sun's 2-step degeneration theory in this context (without the polarized condition) can be directly achieved using results in \cite{Liu} and following the original argument in \cite{DS2}. The essence of the result lies in establishing the algebraic criterion for the polynomial convergence rate. We also remark that, as a consequence of the results in \cite{Colding,CM}, condition \eqref{eq--curvature decay} is equivalent to stating that a tangent cone at $p$ has a smooth link, which is an a priori weaker condition.
   
In this paper, we use the terms ``being conical" and ``polynomially close to a Calabi-Yau cone" interchangeably, which means the following definition holds for some $\delta>0$.
   \begin{definition}\label{def--conical}
 	Let $(Z, J_Z, p)$ be a germ of normal isolated singularity. A K\"ahler metric $(g_Z,J_Z,\omega_Z)$ is said to be conical of order $\delta>0$
 	at $p$, if there exists a Calabi-Yau cone with smooth link $(\ca,o,g_\ca,\omega_\ca,J_\ca)$ and a diffeomorphism $\Phi: (U_o,o)\rightarrow (U_p,p)$ between neighborhoods $U_o$ and $U_p$ of $o$ and $p$ respectively, such that as $r\rightarrow 0$,  for all $k\geq0$, 
  \begin{equation}\label{eq_polyclose}
  	r^k\left(|\nabla^k_{g_{\ca}}(\Phi^*g_Z-g_{\ca})|_{g_{\ca}}+|\nabla^k_{g_{\ca}}(\Phi^*J_Z-J_{\ca})|+|\nabla^k_{g_{\ca}}(\Phi^*\omega_Z-\omega_{\ca})|_{g_{\ca}}\right)=O(r^{\delta}).
  \end{equation} 
 \end{definition}

  It is conjectured in \cite{DS2} that both $W$ and $\ca$ depend only on the germ of the singularity $(Z,p)$, rather than the metric. For algebraic singularities, this conjecture has been confirmed and generalized by Li-Xu and Li-Wang-Xu; see \cite{LX,LWX} and the references therein.
  By this result in algebraic geometry and considering  cones over strictly K-semistable Fano manifolds, we obtain examples $(Z,p)$ that are not conical.
     \begin{corollary}\label{thm-consequence}
  	There are examples of $(Z,p)$, which are pointed Gromov-Hausdorff limits of non-collapsing Calabi-Yau metrics. While the tangent cone $\ca$ at $p$ has a smooth link and is unique, $Z$ is not conical at $p$.
  \end{corollary}

  This highlights the sharpness of the logarithmic rate result proved by Colding-Minicozzi \cite{CM} within the context of K\"ahler geometry, as initially predicted by Hein-Sun \cite{HS}. We remark that in \cite{CM}, the logarithmic convergence rate is also established for tangent cones at infinity, however it is proved in \cite{SZ2023} that every complete Calabi-Yau metric with Euclidean volume growth and quadratic curvature decay is polynomially close to its tangent cone at infinity. This reveals  distinctions between the two settings.

    Then let us compare Theorem \ref{thm-main theorem} with some results in the literature. In \cite{HS}, assuming the germ $(Z,p)$ is biholomorphic to $(\ca,o)$ and $\ca$ has a smooth link, Hein-Sun proved that the $(Z,p)$ is polynomially close $(\mathcal C,o)$ and indeed in \eqref{eq_polyclose}, $\Phi$ can be chosen to be a biholomorphism. In \cite{CSz}, assuming $(Z,p)$ is a polarized limit and the germ $(Z,p)$ is biholomorphic to $(\ca,o)$, Chiu-Sz\'ekelyhidi proved a polynomial rate closeness on the level of K\"ahler potentials without assuming any smoothness on $\ca$.  Although we assume $\ca$ has a smooth link, we are able to address the case where the germ $(Z,p)$ is not biholomorphic to $(\ca,o)$ and provide a characterization for the polynomial closeness. The direction from the polynomial closeness to $\ca=W$ is established directly by constructing holomorphic functions using H\"ormander $L^2$-method. The direction from $\mathcal C=W$ to polynomial closeness uses a approach similar to that in \cite{SZ2023} and relies on the result in \cite{HS,CSz}. See Section \ref{imply polynomial convergence} for more details. 
  \vspace{0.3cm}

  Then we move to singular K\"ahler-Einstein metrics coming from pluripotential theory.  Singular K\"ahler-Einstein metrics with non-positive Ricci curvature on log terminal K\"ahler varieties have been constructed in \cite{EGZ}; see \cite{BBEGZ, BG2014, LiChi} and references therein for further studies. Understanding the metric behavior of the singular K\"ahler-Einstein metrics near the singular locus is a major open problem. The main progress in this direction is due to Hein-Sun \cite{HS}. Following the approach in \cite{HS}, we extend their results to a broader class of singular K\"ahler-Einstein metrics.

  Let $(X,p)$ be a germ of an isolated log terminal algebraic singularity. Motivated by the results in \cite{DS2}, there is a well-developed theory for local stability of (Kawamata) log terminal singularities; see \cite{LL,LX,LWX,Zhuang2023} and the references therein. Specifically, this local stability theory establishes the existence of a K-semistable Fano cone $W$ and a K-polystable Fano cone $\mathcal{C}$, canonically associated with the singularity. Then we have the following.
\begin{theorem}\label{Theorem--local}
Suppose $\ca$ is isomorphic to $W$ as Fano cones and $\ca$ has a smooth link, then every singular  K\"ahler-Einstein metric with non-positive Ricci curvature and with bounded potential on the germ $(X,p)$, is conical at $p$.
\end{theorem}
 

 We remark that as in Theorem \ref{thm-main theorem}, the isomorphism between $\ca$ and $W$ is also a necessary condition for the existence of a conical singular K\"ahler-Einstein metric. If the germ $(X,p)$ itself is biholomorphic to a Calabi-Yau cone, then the assumption in the theorem is satisfied, and we can show that $\Phi$ can be chosen to be a biholomophism. This makes this result a generalization of \cite{HS}. The proof follows the same idea as in \cite{HS}, with some new key insights from  \cite{CW} and \cite{LWX}. Chen-Wang's compactness and regularity results in \cite{CW}, generalizing Cheeger-Colding's results, ensure that we can take limits for K\"ahler-Einstein metrics with mild singularities. Li-Wang-Xu's result in \cite{LWX} ensures that we can determine the tangent cone for the limiting metric from the underlying algebraic structure. See Section \ref{Section--local} for more details. Unfortunately, we are unable to prove similar results for singular K\"ahler-Einstein metrics with positive Ricci curvature. We will briefly discuss the related issues in Section \ref{Section--positive ricci}.

This paper is organized as follows. In Section \ref{preliminary}, we collect some definitions and preliminary results necessary for this paper. In Section \ref{sec--ds theory}, we give a review for Donaldson-Sun's two-step degeneration theory. In Section \ref{polynomial convergence imply}, the direction from the polynomial convergence to $\ca=W$ is proved. In Section \ref{imply polynomial convergence}, the other direction form $\ca=W$ to polynomial convergence is established. In Section \ref{sec-examples}, we construct examples claimed in Corollary \ref{thm-consequence}. In Section \ref{Section--local}, we give the proof of Theorem \ref{Theorem--local}. In the appendices, we give proofs of two results needed in this paper, which are also well-studied in the literature.

   \subsection*{Acknowledgements} 
  The author expresses his gratitude to his advisor Song Sun for the constant support and valuable suggestions. He thanks Xin Fu for discussing the results in \cite{Fu}. The author is partially supported by NSF Grant DMS-2304692. Part of this paper was completed during his visit to IASM at Zhejiang University, which he would like to thank for the hospitality.


  \section{Preliminary}\label{preliminary}

  \subsection{Weighted analysis near the vertex for a cone metric}
  \begin{definition}
  	$(\ca,o,g_{\ca})$ is said to be a (Riemannian) cone with a smooth link if  $\mathcal C\setminus \{o\}=\mathbb R_{>0}\times L$ for some compact manifold $L$ and $g_{\ca}=dr^2+r^2g_L$ with $r$ the projection onto the first factor and $g_L$ is a Riemannian metric on $L$. The point $o$ is referred to as the vertex of the cone.
  \end{definition}
 A cone $\ca$ admits a complete metric $d_{\ca}$, which is the metric completion of the one induced by the Riemannian metric $g_{\ca}$ on $\mathcal C\setminus \{o\}$.  In the following, the symbol $r$ is used to represent the distance function to the vertex of the cone. For every $s>0$, let $B_s$ denote the ball centered at $o$ of radius $s$. For any $k\in \mathbb Z_{\geq 0}$, $\beta\in \mathbb R$ and $\alpha\in (0,1)$, we define the Banach space $C^{k,\alpha}_{\beta}(B_s)$ using the following weighted norms on a function $f$ defined on $B_s\subset\mathcal C$:
\begin{equation}\label{eq--definition of banach space}
	 \|f\|_{C_{\beta}^{k,\alpha}(B_s)}=\sum_{j=0}^k \sup _{B_s}\left|r^{-\boldsymbol{\beta}+j} \nabla^j_{g_\ca} f\right|+\sup_{x\in B_s}\left(r(x)^{-\beta+k+\alpha}[\nabla^k_{g_\ca}f]_{C^{0,\alpha}(B_s\cap B(x,\frac{1}{2}r(x))}\right),
\end{equation}where $[\quad]_{C^{0,\alpha}}$ denotes the ordinary $C^{0,\alpha}$ seminorm.

In the following, we will denote by $\Psi(\epsilon;\cdots)$ any nonnegative functions depending on $\epsilon$ and some additional parameters such that when these parameters are fixed, $\lim_{\epsilon\rightarrow 0}\Psi(\epsilon;\cdots)=0$.
  Now fix $s_0\in (0,1)$. Suppose there is another Riemannian metric $g$ on $B_{s_0}$ such that for all $l\geqslant 0$ and $s\leqslant s_0$,
\begin{equation}\label{eq--new riemannian metric}
    s^l\sup_{B_s}|\nabla^l_{g_{\mathcal C}} (g-g_{\mathcal C})|=\Psi(s;l).\end{equation}
 Fix $k\geqslant 2n+1$ and $\alpha\in (0, 1)$.  Then \eqref{eq--new riemannian metric} ensures that for $s$ sufficiently small
 \begin{equation}
 	\Delta_g: C^{k+2.\alpha}_{2+\beta}(B_s)\rightarrow C^{k,\alpha}_\beta(B_s)
 \end{equation} is a bounded linear operator. Then the following result ensures that $\Delta_g$ admits a right inverse, whose norm is uniformly bounded independent of $s$.
  
   Let $0=\lambda_0<\lambda_1 \leqslant \lambda_2 \leqslant \ldots$ denotes the eigenvalues of the Laplacian on the link $L$ of $\ca$ (listed with multiplicity), and
\begin{equation}\label{eq--exceptional roots}
	\mu_i^{ \pm}=-\frac{m-2}{2} \pm \sqrt{\frac{(m-2)^2}{4}+\lambda_i}.
\end{equation}
Let $\Gamma=\{\mu_i^{\pm}:i=1,\cdots\}$. Then we have

  \begin{proposition} \label{eq--right inverse}
 For any given $\beta\in (0,1)\setminus \Gamma$, there exist $s_1\leqslant s_0$ and $C>0$ such that for all $s\leqslant s_1$, there is a bounded linear map $\mathcal T_s: C^{k,\alpha}_{\beta}(B_s)\to C^{k+2, \alpha}_{\beta+2}(B_s)$ such that $\Delta_{g}\circ \mathcal T_s=\operatorname{Id}$ and $\|\mathcal T_s\|\leqslant C$. 
 \end{proposition}
 
 
\begin{proof}
 This is a standard result and for readers' convenience, we give a sketch of the proof. One first show that $\Delta_{g_{\ca}}$ admits right inverse with norm bounded independent of $s$. For any $s\in (0,1]$ there is a linear extension map 
		$E_s:C^{k, \alpha}_{\beta}(B_s)\rightarrow C^{k, \alpha}_{\beta}(B_2)$
whose norm is bounded independent of $s$. The existence of $E_1$ follows from the local construction in \cite{Seeley}, and then one can define $E_s$ for $s\in (0,1)$ using scaling. Then using separation of variables \cite[Proposition 2.9]{HS}, for $f\in C^{k, \alpha}_{\beta}(B_2)$ one can solve $\Delta_{g_{\ca}}u=f$ with $u$ depending linearly on $f$ and 
\begin{equation*}
	\|u\|_{C^{k,\alpha}_{2+\beta}(B_1)}\leq C \|f\|_{C^{k, \alpha}_{\beta}(B_2)}.
\end{equation*}Then the existence of uniformly bounded right inverse of $\Delta_{g_\ca}$ follows from this.

For a general Riemannian metric satisfying \eqref{eq--new riemannian metric}, one can write $\Delta_{g}=\Delta_{g_\mathcal C}+(g-g_{\mathcal C})*\nabla^2_{g_\mathcal C}+\nabla_{g_{\mathcal C}}g*\nabla_{g_\mathcal C}$. Then acting on function defined on $B_r$, we have
$$\|\Delta_g-\Delta_{g_\mathcal C}\|\leqslant C  \sup_{s\leqslant r, \ 0\leqslant l \leqslant k+2}\Psi(s;l).$$
The right hand side is small when $r \ll 1$. Then the conclusion follows from standard functional analysis. 
\end{proof}

The following result ensures that we can get a point-wise bound of a function from integral bound and the proof follows from the standard elliptic regularity and a rescaling argument. 

\begin{proposition}\label{prop--elliptic estimate}
	Let $g$ be a Riemannian metric on $B_{s_0}$ satisfying \eqref{eq--new riemannian metric}.
	Suppose $u$ is a function defined on $B_{s_0}$ and  for some $\tau, \tau'\in \mathbb R$ and for all $l\geq 0$, 
\begin{equation}
	\left|\nabla^l_g(\Delta_gu)\right|_g\leq C_l r^{\tau-l-2},\text{ and }\int_{B_{s_0}}u^2r^{-2\tau'-2n}d\operatorname{Vol}_g\leq 1.
\end{equation}
 Then for $\gamma=\min(\tau, \tau')$ and any $k>0$, there exist constants $A_{k}>0$ such that
	\begin{equation*}
	    	|\nabla_g^{k}u|_g\leq A_{k} r^{\gamma-k} \text{ on $B_{s_0}$}.
	\end{equation*}
\end{proposition}

  \subsection{Harmonic functions on Calabi-Yau cones.}
  A cone $\ca$ has a natural scaling vector field $r\partial r$. A tensor $T$ on $\ca$ is called $\mu$-homogeneous for some $\mu\in \mathbb R$ if $\mathcal L_{r\partial r} T = \mu T$. We remark that this condition implies that $|\nabla^jT|\sim r^{\mu+p-q-j}$ with respect to the cone metric for all $j \in \mathbb N_0$, if $T$ is $p$-fold covariant and $q$-fold contravariant. In general, if a tensor $T$ satisfies $|T|\sim r^{\lambda}$ as $r\rightarrow 0$, then we will say $T$ has growth rate of $\lambda$.  A K\"ahler cone $\ca$ is a cone with a parallel complex structure $J_\ca$. For a K\"ahler cone $\ca$, the Reeb vector-field $\xi_\ca=J_\ca(r\partial r)$ is a holomorphic Killing vector field. 
  
  \begin{definition}\label{def--calabi-yau cone metric}
	A Calabi-Yau cone with smooth link is a tupe $(\ca, o, g_{\ca}, J_{\ca}, \omega_{\ca})$, such that  
       \begin{itemize}
           \item  $(\ca, o, g_{\ca})$ is a Riemannian cone,
  
  \item $(g_\mathcal C, J_{\mathcal C}, \omega_{\mathcal C})$ defines a K\"ahler Ricci-flat metric on $\mathcal{C}\setminus \{o\}$. 
       \end{itemize}\end{definition}

 The following lemma,
 which is proved in \cite[Theorem 2.14]{HS}, building  on the early work \cite{CT1994}, plays an important role in this paper.  
 \begin{theorem}[\cite{HS}]\label{thm--HS-harmonic functions}
 Let $\ca$ be a Calabi-Yau cone with a smooth link. Then
 \begin{itemize}
 	\item [(1).]If $u$ is a real-valued $\mu$-homogeneous harmonic function with $\mu>0$. Then $\mu\geq1$. If $1<\mu<2$, then $u$ is pluriharmonic. If $\mu=2$, then $u=u_1+u_2$, where $u_1$ and $u_2$ are 2-homogeneous, $u_1$ is pluriharmonic and $u_2$ is $\xi$-invariant.
 	\item [(2).] The space of all holomorphic vector fields that commute with $r\partial r$ can be written as $\mathfrak p\oplus J_\ca\mathfrak p$ where $\mathfrak p$ is spanned by $r\partial r$ and by the gradient fields of the $\xi$-invariant 2-homogeneous harmonic functions. All elements of $J_\ca \mathfrak p$ are Killing fields.
 \end{itemize}
 \end{theorem}

\subsection{Polarized affine cones and Fano cones}\label{sec-terminologies}

In this section, we fix some basic definitions following \cite{CS,LX}. Let $X$ be a normal affine variety with coordinate ring $R(X)$, $\xi$ be a non-zero (real) holomorphic vector field on $X$ which generates a compact torus of automorphisms $\mathbb T$ on $X$ with a unique a fixed point. Then we have a weight decomposition
 \begin{equation}\label{eq-weight decomposition}
 	R(X)=\bigoplus_{\chi \in \Gamma} R_\chi(X).
 \end{equation} Here $\Gamma\subset \operatorname{Lie}(\mathbb T)^*$ is the weight lattice of the $\mathbb T$-action, that is $R_\chi(X)\neq \{0\}$  and $f\in R_\chi(X)$ if and only if $\mathcal L_{V_{\eta}}f=\ii\left<\chi,\eta\right>f$ for every $\eta\in \operatorname{Lie}(\mathbb T)$. Here we use $V_{\eta}$ to denote the holomorphic vector field induced by $\eta$ and $\left<\cdot,\cdot\right>$ to denote the natural pairing between $\operatorname{Lie}(\mathbb T)$ and its dual.
 
 \begin{definition}
 	A polarized affine cone is a pair $(X,\xi)$ such that $\left<\chi,\xi\right>>0$ for all nonzero $\chi\in \Gamma$. Such a $\xi$ is called a Reeb field or polarization. 
 \end{definition}

\begin{definition}
	A polarized affine cone $(X,\xi)$ is said to be a Fano cone, if it has log terminal singularities.
\end{definition}


\begin{remark}
	We use the terminology \emph{polarized affine cone} instead of polarized affine variety as in \cite{CS} to emphasize the grading structure on the coordinate ring $R(X)$. It is shown in \cite{CS} that if $X$ has only an isolated singularity, then a Fano cone is equivalent to a polarized affine cone $(X,\xi)$ with a trivializing section $\Omega$ of $mK_X$ for some $m>0$ such that 
	\begin{equation}\label{eq-lie derivative}
		\mathcal L_{\xi}\Omega=\ii \lambda \Omega
	\end{equation}for some $\lambda>0$.
\end{remark}

A Calabi-Yau cone with smooth link naturally admits a Fano cone structure, as we will briefly review here. For a detailed proof, please refer to \cite{van2011,DS2}. The Reeb vector field $\xi=J_{\ca}(r\partial_r)$ is a holomorphic Killing vector field on $\ca$, which generates a compact torus $\mathbb T$ action on $\ca$.
 We have the following
\begin{itemize}
	\item [(1).]  Polynomial growth holomorphic functions form a finitely generated ring $R(\ca)$ which defines a normal affine variety structure on $\ca$. Moreover there is a positive grading induced by the action of $\mathbb T$, on the coordinate ring 
	\begin{equation*}
R(\ca)=\bigoplus_{\mu\in \mathcal S}R_{\mu}(\ca)
	\end{equation*}where $\mathcal{S}= \mathbb{R}_{\geq 0} \cap\left\{\langle\alpha, \xi\rangle \mid \alpha \in \Gamma\right\}$ and  $\Gamma \subset \operatorname{Lie}(\mathbb{T})^*$  is the weight lattice of the $\mathbb{T}$-action. $\mathcal S$ is a discrete set and is called the \emph{holomorphic spectrum} of $\ca$. In the decomposition, $f\in R_{\mu}(\ca)$ if and only if $f$ is $\mu$-homogeneous, i.e., $\mathcal L_{\xi}f=\ii \mu f$. For $f\in R(\ca)$, its degree $\deg(f)$ is defined to be the smallest number $d$ such that such that $f\in \bigoplus_{\mu\leq d}R_{\mu}(\ca)$ and it is direct to show that this is the same as the growth order of $f$ at infinity.
	\item[(2).] Since $\ca$ has a smooth link, then there exists a $\mathbb T$ invariant, nowhere vanishing holomorphic section $\Omega$ of $mK_{\ca}$ for some positive integer $m$ such that 
	\begin{equation}\label{eq-equation on the cone}
		\omega_{\ca}^n=(c_{n}\Omega\wedge\ols{\Omega})^{\frac{1}{m}}.
	\end{equation}For this, note that the metric on the link $L$ of $\ca$, has positive Ricci curvature and hence $\ca\setminus\{o\}$ has finite fundamental group and therefore the flat bundle $mK_\ca$ admits a parallel section for some $m\geq 1$ as the line bundle $K_\ca$ admits a flat connection induced by the Calabi-Yau metric. Then we can adjust the constant to guarantee \eqref{eq-equation on the cone} holds. Since $\omega_\ca^n$ has finite volume in a neighborhood of $o$, \eqref{eq-equation on the cone} implies that $\ca$ has log terminal singularity at $o$ \cite[Lemma 6.4]{EGZ}.

 \end{itemize}

\begin{remark}
	  Let $(Z,p,d_Z)$ be a pointed Gromov-Hausdorff limit of a sequence of complete non-collapsing K\"ahler metrics  with Ricci curvature uniformly bounded below, and let $(\ca,o)$ be a tangent cone of $Z$ at $p$. In this general setting (without assuming polarized condition and only assuming Ricci curvature lower bound), it is proved in \cite{LS} that $\ca$, which typically has non-isolated singularities, always admits a polarized affine cone structure.
\end{remark}

\subsection{K-stability for Fano cones} General definitions and systematic studies for (log) Fano cone singularities are available, as discussed in works such as \cite{CS, li2021, LX, LWX, Zhuang2023} and references therein. To streamline our discussion and avoid excessive terminology, we adhere to the following definitions, which coincide with those established in the existing literature. 
\begin{definition}
	A Fano cone $(\ca,\xi)$ is said to be K-polystable if it admits a Calabi-Yau cone structure.
	\end{definition} Here in general, a Calabi-Yau cone structure means that there exists a smooth radius function $r$, i.e., $\mathcal L_{-J\xi}r=r$ over $\ca^{reg}$ that is locally bounded on $\ca$, and on $\ca^{reg}$ the K\"ahler form  $\omega=\ii\partial\pp r^2$ satisfies 
	\begin{equation*}
		\operatorname{Ric}(\omega)=0, \quad \text{and}\quad \mathcal L_{-J\xi}\omega=2\omega.
	\end{equation*}When $\ca$ has only an isolated singularity, this coincides with the one given in Definition \ref{def--calabi-yau cone metric}. 
	
	\begin{definition}
	A Fano cone $(X,\xi_X)$ is said to be K-semistable if it admits an equivariant degeneration to a K-polystable Fano cone $(\ca,\xi)$.
\end{definition}
Let $D$ be a Fano manifold and consider the affine variety $$K^{\times}_D:=\operatorname{Spec}(\bigoplus_{m\geq 0}H^0(D,-mK_{D})),$$which is obtained by blowing down the zero section of $K_D$ geometrically. Then $K_D^{\times}$ admits a standard Reeb vector field $\xi$ induced from the fiberwise rotation on $K_D$. It is direct to show that $(K_D^{\times},\xi)$ is K-polytable if and only if $D$ admits a K\"ahler-Einstein metric.

\subsection{H\"ormander $L^2$-method}
The following version of H\"ormander $L^2$-estimate is used and we refer to \cite[Chapter VIII]{Demailly} for a proof:
\begin{theorem}\label{thm--Hormander}
   Let $M$ be an $n$ dimensional complex manifold, which admits a complete K\"ahler metric. Let $\omega_M$ be any  K\"ahler metric on $M$ and $\varphi_M$ be a smooth function with $\sqrt{-1}\partial\pp\varphi_M+\Ric(\omega_M)\geq \Upsilon\omega_M$ for a continuous non-negative function $\Upsilon$. Let $q$ be a positive integer. Then for any $(0,q)$ form $\eta$ on $M$ with  $\pp \eta=0$ and $\int_{M}\Upsilon^{-1} |\eta|_{\omega_M}^2e^{-\varphi_M}\omega_M^n<\infty$, there exists a $(0, q-1)$ form $\zeta$ on $M$ satisfying $\pp\zeta=\eta$ and with estimate
    \begin{equation*}
        \int_{M} |\zeta|_{\omega_M}^2e^{-\varphi_M}\omega_M^n\leqslant \int_{M}q^{-1}\Upsilon^{-1}|\eta|_{\omega_M}^2e^{-\varphi_M}\omega_M^n.
    \end{equation*}
\end{theorem}

\section{Donaldson-Sun two-step degeneration theory}\label{sec--ds theory}
Let $(Z,p,d_Z)$ be a pointed Gromov-Hausdorff limit of complete non-collapsing K\"ahler-Einstein manifolds with uniformly bounded Ricci curvature. This means that there is a sequence of complete K\"ahler manifolds $(X_i,h_i)$ and points $q_i\in X_i$ satisfying $\Ric(h_i)=\lambda_i h_i$ with $|\lambda_i|\leq 1$ and $\vol(B(q_i,1))\geq \kappa$ for some constant $\kappa$ independent of $i$, such that $(X_i,q_i,h_i)$ converges to $(Z,p,d_Z)$ in the pointed Gromov-Hausdorff topology. Indeed, since we are interested in the local geometry of $Z$ near $p$, completeness of $(X_i,h_i)$ is not important as long as $B(q_i,1)$ is compactly contained in $X_i$.

  Note that we do not assume the K\"ahler-Einstien metrics are polarized, therefore Donaldson-Sun's theory can not directly applied here.
   We assume the following curvature condition holds near $p$:
 \begin{equation}\label{eq-curvature decay}
 	|\Rm(\cdot)|\leq Cd_Z(p,\cdot)^{-2}.
 \end{equation} 
Liu's result \cite[Proposition 2.5]{Liu} is the crucial step in establishing Donaldson-Sun's two-step degeneration theory in our context. The curvature condition \eqref{eq-curvature decay} is needed there to obtain point-wise bound from integral bound for functions. It's worth noting that although Liu's work primarily concerns tangent cones at infinity, the proof can be directly applied here without much modification. 

Let $(Z_j,p_j,J_j, g_j, \omega_j,d_j)$ denote the rescaling of $(Z,p,J_Z,g_Z,\omega_Z,d_Z)$ by a factor $2^{j}$, $j\in \mathbb Z_{\geq 1}$.  This means that we take $J_j=J, p_j=p, g_j=2^{2j}g_Z, \omega_j=2^{2j}\omega_Z$ and $d_j=2^j d_Z$. Let $r_j$ be the distance to $p_j$ with respect to $d_j$.  Recall that we denote by $\Psi(\epsilon;\cdots)$ any nonnegative functions depending on $\epsilon$ and some additional parameters such that when these parameters are fixed, $\lim_{\epsilon\rightarrow 0}\Psi(\epsilon;\cdots)=0$.
 
 \begin{proposition}[\cite{Liu}] \label{prop--Liu} Given any $r_0>0$, for sufficiently large $j$, we can find a plurisubharmonic function $u$ on $B\left(p_j, r_0\right)$ with
$$
\begin{gathered}
\left|u-r_j^2\right| \leq \Psi(j^{-1}) r_0^2,\\
|\nabla u|^2-4 r_j^2 \leq \Psi(j^{-1}) r_0^2\\
\ii\partial\pp u\geq (1-\Psi(j^{-1}))\omega_j \text{ on $B(p_j, r_0)\setminus B(p_j,\Psi(j^{-1})r_0)$} .
\end{gathered}
$$
 \end{proposition}
 
 Proposition \ref{prop--Liu} implies that there are (arbitrary small) neighborhoods of $p$ in $Z$, which have strictly pseudoconvex boundary. In the regular part of $Z$, we have smooth convergence of K\"ahler-Einstein metrics $(X_i, q_i,h_i)$, therefore there are neighnorhoods of $q_i$ with strictly pseudoconvex boundary, converging in the Gromov-Hausdorff sense. Then we can use H\"ormander $L^2$ estimate to construct holomorphic functions in a neighborhood of $q_i$, therefore repeating the argument in \cite{DS1,DS2}, one can show that $Z$ has a normal complex analytic space structure. 
 
  Let $(\ca, o, J_\ca,g_\ca, \omega_\ca, d_\ca)$ denote a tangent cone of $(Z,p, g_Z, J_Z,\omega_Z, d_Z)$ at $p$. As a consequence of \eqref{eq-curvature decay}, $\ca$ is a Calabi-Yau cone with a smooth link. Since $Z$ admits a normal analytic space structure with an isolated singularity $p$, we can also do H\"ormander $L^2$-method in a neighborhood of $p\in Z$. In particular, we have the following  \cite{DS1}.
\begin{proposition}
		For any $\lambda>0$, and any holomorphic function $f$ defined on $B(o,\lambda)\subset \ca$, for $j$ sufficiently large, there exists holomorphic functions $f_j$ defined on $B(p_j,\lambda/2)$ such that $f_j$ converges uniformly to $f$.
\end{proposition}

Donaldson-Sun's two-step degeneration theory can be established following their original argument \cite{DS2}. In particular, the tangent cone $\ca$ is unique. This theory describes the Gromov-Haursdorff convergence of $Z_j$ to $\ca$ using complex analytic data and we summarize their results as follows.

Let $B_j$ denote the unit ball in $Z_j$ centered at $p_j$ and $B$ the unit ball in $\ca$ centered at $o$.
  Fix a distance $\mathbf{d}_j$ on $B_j\sqcup B$ that realized the Gromov-Hausdorff convergence of $B_j$ to $B$. More precisely, this means that the Hausdorff distance between $B_j$ and $B$ under $\mathbf d_j$ is $\Psi(j^{-1})$ and $\mathbf d_j(p_j,p)=\Psi(j^{-1})$. Moreover for any compact set $K$ contained in the regular part of $B$, we can find for large enough $j$ open embeddings $\chi_j$ of an open neighborhood of $K$ into $B_j$ such that $\mathbf d_j(x,\chi_j(x))=\Psi(j^{-1})$ for all $x\in K$ and $(\chi_j^*g_j,\chi_j^*J_j)$ converges smoothly over $K$ to $(g_{\ca},J_{\ca})$.

It is proved in \cite{DS2} that there exist holomorphic embeddings 
 \begin{equation}
 	F_{\infty}:(\ca,o) \rightarrow (\mathbb C^N,0) \text{ and } F_j:(B_j,p_j)\rightarrow (\mathbb C^N,0)
 \end{equation}such that 
 \begin{itemize}
 	\item[(1).] $F_{\infty}=(h_1,\cdots,h_N)$, where each $h_i$ is a $d_i$-homogeneous holomorphic function with $d_i>0$. Under the embedding $F_{\infty}$, $\xi=J_{\ca}(r\partial _r)$ extends to a linear vector field on $\C^N$ of the form $Re(\sqrt{-1}\sum_i d_i z_i\partial_{z_i})$, which we also denote by $\xi$. In particular, the dilation action $\Lambda:r\rightarrow 2r $ on $\ca$ extends to the diagonal linear transformation of $\C^N$ given by 
$$\Lambda(z_1,\cdots, z_N)=(2^{d_1}z_1, \cdots, 2^{d_N}z_N).$$

 	\item[(2).] $F_{j}=\Lambda_j\circ F_{j-1}|_{B_{j}}=\Lambda_j\circ \cdots\circ \Lambda_2\circ F_1|_{B_{j}}$, where $\Lambda_j\in G_{\xi}$. Here $G_\xi$ denotes the subgroup of $GL(N;\C)$ consisting of elements that commute with the actions generated by $\xi$. Moreover $\Lambda_j\rightarrow \Lambda$ as $j\rightarrow \infty$.
 	\item[(3).] $\mathbf{d}_j(x_j,x)\rightarrow 0$ if and only if $F_j(x_j)\rightarrow F_{\infty}(x_{\infty})\in \mathbb C^N$. 
 \end{itemize}

 Moreover there is an intrinsic \emph{intermediate K-semistable cone} $W$ associated to $(Z,p)$ and can be characterized as follows.
 Let $\mathcal O_p$ be the ring of germs of holomorphic functions on $Z$ at $p$. For any non-zero function $f\in \mathcal O_p$, one defines its order of vanishing:
 \begin{equation}
 	d_{KE}(f):=\lim _{r \rightarrow 0} \frac{\log \sup _{B(p, r)}|f|}{\log r}.
 \end{equation} Such a limit exist and is in $\mathcal S$, the holomorphic spectrum of $\ca$. We list them in order as $0=d_0< d_1< \cdots $. Then for any $k\geq 0$, one can define an ideal 
 \begin{equation*}
 	I_k=\left\{f \in \mathcal{O}_x: d_{KE}(f) \geq d_k\right\} .
 \end{equation*}
 We obtain a filtration $\mathcal{O}_p=I_0 \supset I_1  \supset \cdots$ and an associated graded ring
$
\bigoplus_{k \geq 0} I_k / I_{k+1}.
$ Then the \emph{intermediate K-semistable cone} is defined to be $$W:=\operatorname {Spec}(\bigoplus_{k \geq 0} I_k / I_{k+1})$$ which is a normal affine variety and also admits a natural $\mathbb T$ action, which defines a polarized affine cone structure on $W$. 

The relation between $W$ and $\ca$ can be described as follows. Let $W_j$ denote the weighted tangent cone of $F_j(B_j)$ with respect to the weight $(d_1,\cdots,d_N)$. 
 It is proved in \cite{DS2} that each $W_j$ is isomorphic to $W$ as polarized affine cones. As $\Lambda_j\in G_{\xi}$, commuting with $\xi$, we have $W_j=\Lambda_j(W_{j-1})$. Moreover $W_j$ converge to $F_{\infty}(\ca)$ in a certain multi-graded Hilbert scheme $\mathbf {Hilb}$. It is shown in \cite{DS2} Transverse automorphisms of the cone $\ca$, i.e., automorphisms of $\ca$ that preserve the Reeb vector field $\xi=J_\ca(r\partial r)$, form a reductive complex Lie group. Then a variant of Luna's slice theorem, proved by Donaldson \cite{Donaldson} implies the existence of a one parameter subgroup $\lambda(t)$ of $G_{\xi}$ such that $\ca=\lim_{t\rightarrow 0}\lambda(t).W$. That is $W$ admits an equivariant degeneration to $\ca$. As $\ca$ has an isolated log terminal singularity and has a quotient singularity in dimension 2, by \cite[Chapter 9]{Ishii}, we know that $W$ also has an isolated log terminal singularity. Therefore $W$ admits a Fano cone structure, justifying the terminology \textit{intermediate K-semistable cone}.

 \section{Polynomial convergence implies $\ca=W$}\label{polynomial convergence imply} Suppose $(g_Z,J_Z,\omega_Z)$ is conical of order $\delta>0$, we show in this section that $\ca $ is isomorphic to $W$ as Fano cones. This is achieved using H\"ormander $L^2$-method 
 to construct $J_Z$-holomorphic functions, whose rescalings converge directly to holomorphic functions on $\ca$.

 Let $\Phi$ denote the diffeomorphism given in the definition \eqref{def--conical}. Using the diffeomorphism $\Phi$, we will identify the neighborhood $U_{p}$ of $p$ with $U_o$, which, without loss of generality, we may assume to be $B$, the unit ball in $\ca$ centered at the vertex. We recall that $d_j$ denotes the rescaling of $d_Z$ by a factor $2^{j}$, and $B_j$ denotes the unit ball in $Z_j$, and $\Lambda$ denotes the dilation $r\rightarrow 2r $ on $\ca$.
Then using \eqref{eq_polyclose}, one can easily show the following result. 

\begin{lemma}
There exist a constant $C>0$ such that for all $x,y \in B$, and $j\in \mathbb N$, $$|d_{j}(\Lambda^{-j}.x,\Lambda^{-j}.y)-d_{\ca}(x,y)|\leq C 2^{-j\delta}.$$ 
\end{lemma}
 As mentioned in \cite{DS2}, the notion of convergence of holomorphic functions depends on the choice of the metric on $B_j\sqcup B$, which realizes the pointed Gromov-Hausodrff convergence. Moreover the argument in \cite{DS2} works for any such a choice. The above lemma implies that 
\begin{equation*}
	\Lambda^{-j}: (B, d_{\ca})\rightarrow (B_j, d_j)
\end{equation*}realizes the pointed Gromov-Hausdorff convergence as $j\rightarrow \infty$. Therefore, in the following, when we talk about the convergence of holomorphic functions, we are considering convergence under this Gromov-Hausdorff approximation.

Let $\psi$ denote the plurisubharmonic function $\log r^2-(-\log r^2)^{\frac{1}{2}}$ on $\ca$. Note that 
\begin{equation}
	\ii\partial_{J_{\ca}}\pp_{J_{\ca}}\psi=\left(1+\frac{1}{2}(-\log r^2)^{-\frac{1}{2}}\right)\frac{\ii\partial\pp r^2-4\ii\partial r\wedge\pp r}{r^2}+\frac{\ii\partial r\wedge\pp r}{r^2(-\log r^2)^{\frac{3}{2}}}.
\end{equation}
Then it follows that 
\begin{equation}\label{eq-weight function}
	\ii\partial_{J_\ca}\pp_{J_{\ca}} \psi\geq \frac{1}{(-\log r^2)^{\frac{3}{2}} r^{2}}\omega_{\ca} \text{ on $\{r^2< e^{-1}\}$}.
\end{equation}
Since $(g_Z, J_Z, \omega_Z)$ is K\"ahler-Einstein and conical of order $\delta>0$, as a consequence of \eqref{eq-weight function}, we know that in a neighborhood of $o$, for any $\epsilon>0$, there exists a constant $c_{\epsilon}>0$ such that
\begin{equation}\label{eq--positivity of psh function}
	\ii\partial_{J_Z}\pp_{J_Z} \psi+\operatorname{Ric}(\omega_Z) \geq \frac{c_{\epsilon}}{r^{2-\epsilon}}\omega_Z.
\end{equation}

There are homogeneous holomorphic functions $(h_1,\cdots,h_N)$ on $\ca$, generating $R(\ca)$ and giving an embedding 
\begin{equation*}
	F=(h_1,\cdots,h_N):\ca\rightarrow \mathbb C^N
\end{equation*} such that the Reeb vector field $\xi$ extends to a linear diagnoal vector field $\Re(\sqrt{-1}\sum_i d_i z_i\partial_{z_i})$ and the dilation $\Lambda$ on $\ca$ extends to a diagonal dilation on $\C^N$,
\begin{equation*}
	\Lambda(z_1,\cdots,z_N)=(2^{d_1}z_1,\cdots,2^{d_N}z_N),
\end{equation*}where $d_i=\deg(h_i)>0$. 

Then as $(g_Z, J_Z, \omega_Z)$ is conical of order $\delta>0$, we obtain that for all $l\geq 0$,
\begin{equation}\label{eq--decay of right hand side}
	|\nabla_{\omega_Z}^l(\pp_{J_Z}h_i)|_{\omega_Z}=O(r^{d_i-1-l+\delta}).
\end{equation}
 We choose a neighborhood $U$ of $o$ with smooth strictly pseudoconvex boundary with respect to the complex structure $J_Z$ such that \eqref{eq--positivity of psh function} holds on $U$, and choose the weight $\psi_i=(d_i+\frac{\delta}{2}+n)\psi$ and the background K\"ahler metric $\omega_Z$. Using Theorem \ref{thm--Hormander},  we can solve the $\pp_{J_Z}$-equation
\begin{equation}\label{eq-partial bar equation}
	\pp_{J_Z} u_i=\pp_{J_Z}h_i,
\end{equation}and get a solution $u_i$ of \eqref{eq-partial bar equation} with the estimate
\begin{equation}\label{eq-estimate on solution}
	\int_{U} |u_i|^2e^{-(d_i+\frac{\delta}{2}+n)\psi}\omega_Z^n<\infty.
\end{equation}
Let $f_i=h_i-u_i$.
Then we have the following.
\begin{lemma}
	For any $i\in \{1,\cdots,N\}$, $2^{jd_i}f_i(\Lambda^{-j}.x)\rightarrow h_i(x)$ uniformly as $j\rightarrow \infty$.
\end{lemma}
\begin{proof}
	As $h_i$ is homogeneous of degree $d_i$, $2^{jd_i}h_i(\Lambda^{-j}.x)=h_i(x)$. To prove the lemma, it suffices to show that $2^{jd_i}u_i(\Lambda^{-j}.x)\rightarrow 0$. This follows from \eqref{eq--decay of right hand side},\eqref{eq-estimate on solution} and Proposition \ref{prop--elliptic estimate}.
\end{proof}
 This lemma says that with respect to the Gromov-Hausdorff approximation $\Lambda^{-j}$, holomorphic function $2^{jd_i}f_i|_{B_j}$ converges to $h_i|_{B}$ on $\ca$ as $j\rightarrow \infty$. Therefore it implies that $\ca=W$ in Donaldson-Sun's two-step degeneration theory.

 \section{$\ca=W$ implies polynomial convergence}\label{imply polynomial convergence}

 The overall idea of the proof is similar to \cite{SZ2023}. Firstly using $\ca=W$, we get an almost K\"ahler-Einstein metric $(\tilde \omega,J_{\ca}, \tilde g)$ in a neighborhood of $o\in \ca$, which is (by construction) polynomially close to the singular K\"ahler-Einstein metric $(\omega_Z, J_Z, g_Z)$. Then solving a complex Monge-Amp\`ere equation, we get a Calabi-Yau metric $(\bar \omega, J_\ca, \bar g)$ in a neighborhood of $o\in \ca$, which is again polynomially close to the singular K\"ahler-Einstein metric $(\omega_Z, J_Z, g_Z)$. Finally we can apply Chiu-Szekelyhidi's result \cite{CSz} to get that $(\bar \omega, J_\ca, \bar g)$ is polynomially close to the Calabi-Yau cone metric $(\omega_{\ca}, J_{\ca}, g_{\ca})$.

  We follow the notations in Section \ref{sec--ds theory}. Moreover in the following we identify $\ca $ with its image $F_{\infty}(\ca)\subset \mathbb C^N$, and as we only care about the local geometry near $p$, we also identify $Z$ with its image $F_1(B_1)\subset \mathbb C^N$.

\subsection{Construction of the diffeomorphism}

  The following  lemma is proved in \cite[Lemma 3.6]{HS}. We include its proof for readers' convenience. Recall that $G_\xi$ denote the subgroup of  $GL(N;\C)$ consisting of elements that commute with the actions generated by $\xi$. Let $G$ denote the subgroup of $G_\xi$, which leaves $\ca$ invariant. 
    \begin{lemma}[\cite{HS}]\label{lemm--W is the same}
 	If $W$ is isomorphic to $\ca$ as polarized affine cones, then in Donaldson-Sun's 2-step degeneration theory we can make $W_j$ equal to $\ca$ for all $j$.
 \end{lemma}
 
 \begin{proof}
 	Recall that we identify $\ca$ with $F_{\infty}(\ca)\subset \mathbb C^N$ and both $W_j$ and $F_{\infty}(\ca)$ are in a Hilbert scheme $\mathbf {Hilb}$. Since $W$ is isomorphic to $\ca$ and $W_j$ is isomorphic to $W$ as polarized affine cones, there exists $g_j\in G_{\xi}$ such that  $W_j=g_j(\ca)$. Moreover $W_j$ converges to $\ca$ in $\mathbf{Hilb}$, then a variant of Luna's slice theorem, proved by Donaldson \cite{Donaldson} implies that there exist $h_j\in G_{\xi}$ with $h_j\rightarrow \operatorname{Id}$ and $h_j(W_j)=h_j(g_j(\ca))=\ca $. Replacing $F_j$ with $h_j\circ F_j$, we can therefore guanrantee that $W_j=\ca$.
 \end{proof}

  The following lemma can be proved in a similar way as in \cite[Proposition 3.10]{SZ2023}. It plays an essential role for constructing K\"ahler metrics by pulling back their potentials.  We fix a K\"ahler cone metric $\omega_{\xi}$ on $\mathbb C^N$ with the Reeb vector field given by $\xi$. Such a metric always exists; see \cite[Lemma 2.2]{HeS}. Although when restricted to $\ca$, $\omega_{\xi}$ is different from $\omega_{\ca}$, they share the same Reeb vector field. This implies that $\omega_{\ca}$ and $\omega_{\xi}$ are uniformly comparable and they define the same Banachs space using \eqref{eq--definition of banach space} with comparable norms on it.  Denote by $r_\xi = d_{\omega_\xi}(0,·)$ the radial function on $\mathbb C^N$ defined by $\omega_{\xi}$.
  
 \begin{lemma}\label{lemma: complex structure close}
 	There exists a diffeomorphism $\Phi$ from  a neighborhood of $o\in \ca$ to a neighborhood of $p\in Z$, such that, for some $\delta_0>0$ and all $k\geq 0$, as $r_{\xi}\rightarrow 0$
\begin{equation*}
	|\nabla^k_{\omega_{\xi}}(\Phi^*J_{Z}-J_{\ca })|_{\omega_{\xi}}=O(r_{\xi}^{\delta_0-k}).
\end{equation*}  
 \end{lemma}
 
 \begin{proof}
 	Recall that we identify a neighborhood of $p\in Z$ with a neighborhood of $p\in F_1(B_1)\subset \mathbb C^N$. We define $\Phi^{-1}$ to be the normal projection map to $\ca$, i.e., for $x$ in a neighnorhood of $p$, we let $\Phi^{-1}(x)$ to the unique point in $\ca $ that is closest to $x$ with respect to the metric $\omega_\xi$. We need to show that this is well-defined in a neighborhood of $p\in Z$ and satisfies the desired properties.

 	As shown in Section \ref{sec--ds theory}, $Z$ is a normal complex analytic space with an isolated singularity at $p$. Then we know that there exists holomorphic functions $f_1, \cdots,f_m$ defined in a neighborhood of $0\in \mathbb C^N$ and a neighborhood $V$ of $p$ in $Z$ such that $V$ is the common zero of $f_i$ and $\operatorname{rank}(df_1,\cdots,df_m)=N-n$ on $V\setminus {p}$. Let $g_i$ be the initial term of $f_i$
 with respect to the weight induced by $\xi$, i.e. the homogeneous term of smallest weight in the Taylor expansion of $f_i$. We may assume $f_i$ generate the ideal of $\ca$, which is equal to the weighted tangent cone of $(F_1(B_1),p)$ with respect to $\xi$.  Let $w_i$ denote the weight of $f_i$.
 	
 	For $l\geqslant 1$  we denote by $A_l$ the annulus in $\C^N$ defined by $2^{-l}\leqslant r_\xi\leqslant 2^{-l+1}$. By the conical nature of $\omega_\xi$, it suffices to consider the normal projection from $\Lambda^l(F_1(B_1)\cap A_{l+1})\subset A_1$ to $\ca $ for $l\gg1$.  
 Then the defining functions of $Y_l:=\Lambda^l(F_1(B_1)\cap A_{l+1})\subset A_1$ can be chosen to be 
 	\begin{equation}
 		f_{i,l}(z)=2^{w_il}f_i(\Lambda^{-l}.z).
 	\end{equation} 
 	As the weights of $\xi$ form a discrete set and $g_i$ is the initial term of $f_i$
 with respect to this weight, we know that there exist some holomorphic functions $e_{i l}(z)$ with $\sup_{|z|\leq 1}|e_{i,l}(z)|$ uniformly bounded independent of $i$ and $l$ and constants $\delta_{i}> 0$, such that 
 \begin{equation}
 	f_{i,l}(z)=g_i(z)+2^{-l \delta_i}e_{i,l}(z). 
 \end{equation} Let $\delta_0=\min_i \delta_i. $
 
\begin{claim}
 For $l$ sufficiently large, there are finitely many open sets $U_{\gamma}\subset \mathbb C^N$ such that $\ca\cap U_{\gamma}$ is a cover of $\ca$ and for each $\gamma$, $ Y_l\cap U_{\gamma}$ is a graph of holomorphic functions over $\ca\cap U_{\gamma}$, that is there are holomorphic functions $h_{\gamma,l}: U_\gamma\cap W\rightarrow\C^N$ with $|\nabla^k_{\omega_\xi}(h_{\gamma, l}-\operatorname{id})|_{\omega_{\xi}}\leqslant C_k 2^{-l\delta_0}$ for all $k\geqslant 0$ and  $Y_l\subset \cup_\gamma \text{Im}(h_{\gamma,l})$.
\end{claim}

To prove the \textbf{Claim}, we fix a point $z\in \ca \cap A_1$. Then we can find a neighborhood $U\subset \C^N$, such that $\ca \cap U$ is given by the zero set of $N-n$ number of  $g_\alpha$'s and $\operatorname{rank}(df_1,\cdots,df_m)=N-n$ on $U$. For simplicity of notation we may assume these are $g_1, \cdots g_{N-n}$. Shrinking $U$ if necessary and using implicit function theory, we can find local holomorphic coordinates $\{\zeta_1, \cdots, \zeta_N\}$ such that $\zeta_\alpha=g_\alpha$ for $\alpha=1, \cdots N-n$. Now we have  $f_{\alpha, l}=g_\alpha(z)+2^{-l\delta_\alpha} e_{\alpha l}(z)$.  It follows that for $l$ large the common zero set of $\{f_{\alpha, l}\}_{\alpha=1}^{ N-n}$ is a smooth complex submanifold in $U$, so in particular it agrees with $Y_l\cap U$, by shrinking $U$ is necessary. Using the local  coordinates $\{\zeta_1, \cdots, \zeta_N\}$ and implicit function theorem, it is easy to see that $Y_l\cap U$ is contained in the image of a holomorphic function $h_l:U\rightarrow \C^N$ such that $|\nabla^k_{\omega_\xi}(h_l-\operatorname{id})|_{\omega_\xi}\leqslant C_k2^{-l\delta_0}$ for all $k\geqslant 0$. Since $A_1$ is compact the \textbf{Claim} follows.  

Since the normal injectivity radius of $\ca \cap A_1$ is uniformly bounded, there is a tubular neighborhood $\mathcal N$ of $\ca \cap A_1$ such that the normal projection map $\Pi: \mathcal N\rightarrow \ca$ is smooth.  It follows from the \textbf{Claim} that for $l$ large, $Y_l\subset \mathcal N$.  So $\Phi$ is well-defined and smooth in a neighborhood of $p$. It is straightforward to check that it satisfies the desired derivative bounds.   The estimates on the complex structures follow from the fact that $h_{\gamma, l}$ is holomorphic and $\omega_{\xi}$ is a K\"ahler metric.
 \end{proof}

\subsection{Construction of the almost K\"ahler-Eisntien metrics}
Using the uniqueness of the tangent cone at $p\in Z$, one can show that, see for example \cite[Lemma 3.1]{SZ2023}, a diffeomorphism $\Phi_0: U_o\rightarrow U_p$ such that for all $k\geq 0$, 
    \begin{equation}\label{eq:weak estimates}
    	 \lim_{s\rightarrow 0}\sup_{B_s}s^k\left (\left|\nabla^k_{g_\mathcal C}(\Phi_0^*g_Z-g_{\mathcal C})\right|+\left |\nabla^k_{g_\mathcal C}(\Phi_0^*J_Z-J_{\mathcal C})\right|+\left |\nabla^k_{g_\mathcal C}(\Phi_0^*\omega_Z-\omega_{\mathcal C})\right|\right)=0.
    \end{equation}
In particular, we know that on a neighborhood of $p$ there exists a smooth function $r_Z$ which is comparable to the distance function $d_Z(p,\cdot)$. Recall that we use $r$ to denote the radial function on $\ca$ defined by the Calabi-Yau cone metric and $r_\xi$ to denote the radial function on $\mathbb C^N$ defined by $\omega_\xi$, and we know that $r$ and $r_\xi$ are comparable on $\ca$.


Repeating the proof of \cite[Proposition 3.5]{SZ2023}, we get the following rough estimates. For simplicity of notations, we omit the pullback notation and the restriction notation, and naturally view $\omega_Z$ and $\omega_{\xi}$ as K\"ahler forms on $F_1(B_1)$ in the following.
 \begin{lemma}\label{lemma:rough estimate}
In a neighborhood of $p\in F_1(B_1)$, for all $\delta>0$ and $k\geqslant 1$, we have	\begin{equation}\label{e:rough comparison} C_\delta^{-1} r^{\delta}\omega_Z\leqslant \omega_\xi\leqslant C_\delta r^{-\delta}\omega_Z, 
 \end{equation}
 \begin{equation}\label{eqn4.3}|\nabla^k_{\omega_Z}\omega_{\xi}|_{\omega_Z}\leqslant C_{\delta,k} r^{-\delta-k}, 
 \end{equation}
 \begin{equation}\label{e:rough comparison of distance}
 C_\delta^{-1} r_Z^{1+\delta}\leqslant r_\xi\leqslant C_\delta r_Z^{1-\delta}.	
\end{equation}
  \end{lemma}
 
 Note that we do not claim the uniform equivalence between $\omega_Z$ and $\omega_\xi$. Nonetheless we know the error is smaller than any polynomial order and this suffices for our applications. Then we show that we can find a K\"ahler potential for $\omega_Z$ with an almost quadratic order estimate.

\begin{lemma}\label{lemma--potential control}
	For any $\epsilon>0$, there exists a function $\psi_{\epsilon}$ in a neighborhood of $p$ such that $\omega_Z=\ii \partial \pp\varphi_{\epsilon}$ and $|\nabla_{\omega_Z}^k\varphi_{\epsilon}|=O(r_Z^{2-\epsilon-k})$ for all $k\geqslant 0$.
\end{lemma}

\begin{proof}

 	Firstly, we show that $\omega_Z$ is a $d$-exact 2-form in a neighborhood of $p$. As a consequence of \eqref{eq:weak estimates}, we know that a neighborhood of $p$ in $F_1(B_1)$ is diffeomorphic to $(0,1)\times Y$ some smooth compact manifold $L$ and $\omega_Z$ is a closed 2-form with norm bounded with respect to the cone metric $g_\ca=d r^2+r^2g_L$.
 	Note that $\omega_Z$ defines a cohomology class $\alpha$ in $H^2(L,\mathbb R)$. The fact that the norm $\omega_Z$ is bounded with respect to $g_{\ca}$, implies that in the class $\alpha$, we can find smooth representatives which is $O(r^{2})$ with respect to the fixed Riemannian metric $g_L$. Therefore this cohomology class $\alpha$ has to be $0$. Therefore we know that $\omega_Z$ is a $d$-exact 2-form in a neighborhood of $p$.
	
	Secondly, we show that $\omega_Z=d\eta$ for some real 1-form $\eta$ with 
	\begin{equation*}
		|\nabla_{\omega_Z}^k\eta|=O(r_Z^{1-k}) \text{for all $k\geq 0$}.
	\end{equation*}
	This is a standard integration argument. Writing $\omega_Z=dr\wedge \alpha_1+\alpha_2$ with $\partial_r\lrcorner \alpha_1=\partial_r\lrcorner \alpha_2=0$. Then we define $\eta = \int_{0}^r \alpha_1 dr$. Using the fact that $\omega_Z$ is $d$-exact, one can directly check that $\omega_Z=d\eta$. Then the estimate follows from \eqref{eq:weak estimates}.

 Let $\psi$ denote the plurisubharmonic function $\log r_\xi^2-(-\log r_\xi^2)^{\frac{1}{2}}$ on $\mathbb C^N$. Then by Lemma \ref{lemma:rough estimate}, we know that in a neighborhood of $o$, for any $\epsilon>0$, there exists a constant $c_{\epsilon}>0$ such that
 \begin{equation}
 	\ii\partial_{J_Z}\pp_{J_Z} \psi\geq \frac{c_\epsilon}{r^{2-\epsilon}}\omega_{Z}.
 \end{equation}
By Theorem \ref{thm--Hormander}, choosing the weight as $(n-\epsilon+2)\psi$ and using $\omega_Z$ is K\"ahler-Einstein, we can solve the $\pp$ equation with integral estimate for $\varphi_\epsilon$
 \begin{equation*}
 	\pp_{J_Z}\varphi_\epsilon=\eta^{0,1}.
 \end{equation*}  As in Section \ref{polynomial convergence imply}, using Proposition \ref{prop--elliptic estimate}, we can get pointwise estimate for $\varphi_\epsilon$.
 \end{proof}

Let $\Phi$ and $\delta_0$ be the diffeomorphism and the constant obtained in Lemma \ref{lemma: complex structure close}. Choosing $\epsilon=\frac{\delta_0}{2}$ in Lemma \ref{lemma--potential control}, we obtain a K\"ahler potential $\varphi=\varphi_{\epsilon}$. Then we can define a K\"ahler metric $(\tilde g,\tilde \omega,J_{\ca})$ in a neighborhood of $o$ by 
\begin{equation}
	\tilde\omega=\ii\partial_{J_{\ca}}\pp_{J_{\ca}}\Phi^*(\varphi), \quad \tilde g(\cdot,\cdot)=\tilde\omega(\cdot,J_{\ca}\cdot).
\end{equation}
As a consequence of Lemma \ref{lemma: complex structure close}, Lemma \ref{lemma:rough estimate} and Lemma \ref{lemma--potential control}, we obtain that in a neighborhood of $o$, there exists $\delta_1>0$ such that for all $k\geq 0$, we have
	\begin{equation}\label{eq--estimate for tilde}
		|\nabla^k_{\tilde g}(\Phi^*\omega_Z-\tilde \omega)|=O(r^{\delta_1-k})\quad \text{and} \quad
		|\nabla^k_{\tilde g}(\Phi^*g_Z-\tilde g)|=O(r^{\delta_1-k}).
	\end{equation}

\subsection{Solving complex Monge-Amp\`ere equations locally}
Next we are going to solve some complex Monge-Amp\`ere equations via Banach fixed point theorem to construct Calabi-Yau metrics in a neighborhood of $o\in \ca$. For this we need to specify the Banach spaces that we are working on. Combining \eqref{eq:weak estimates}, Lemma \ref{lemma: complex structure close} and the construction of $(\tilde g, \tilde \omega)$, we know that there exists a self-diffeomorphism $\Phi_1$ of $U_o$ such that 
\begin{equation}\label{eq: needed for Banach spaces}
    	 \lim_{s\rightarrow 0}\sup_{B_s}s^k\left (\left|\nabla^k_{\tilde g}(\tilde g-\Phi_1^*g_{\mathcal C})\right|+\left |\nabla^k_{\tilde g}(J_{\ca}-\Phi_1^*J_{\mathcal C})\right |+\left|\nabla^k_{\tilde g}(\tilde \omega-\Phi_1^*\omega_{\mathcal C})\right|\right)=0.
    \end{equation}
In the following, Banach spaces are defined with respect to the cone metric $\Phi_1^*g_{\mathcal C}$.

\begin{proposition}\label{prop--solve ma to get calabi-yau}There exists a $\delta_2>0$ such that 
	for sufficiently small $s$, there exists $\varphi\in C^{k,\alpha}_{2+\delta_2}(B_s)$ such that 
	\begin{equation}
		\bar \omega=\tilde\omega+\ii\partial\pp \varphi
	\end{equation}
 is a Calabi-Yau metric.
\end{proposition}

\begin{proof}  
Since $\omega_Z$ is K\"ahler-Einstien, i.e. $\operatorname{Ric}(\omega_Z)=\lambda \omega_Z$ for some $\lambda\in \mathbb R$. By \eqref{eq--estimate for tilde}, we know that 
\begin{equation*}
	\left|\nabla^k_{\tilde g}\operatorname{Ric}\tilde \omega\right|=O(r^{\delta_1-2-k}).
\end{equation*} Then by a similar argument to Lemma \ref{lemma--potential control}, we obtain that there exists a smooth function $\tilde f$ such that $\operatorname{Ric}(\tilde \omega)=\ii\partial\pp \tilde f$ and $|\nabla_{\tilde g}^k \tilde f|=O(r^{\frac{\delta_1}{2}-k})$. Then we want to solve
\begin{equation}
	(\tilde\omega+\ii\partial\pp \varphi)^n=e^{\tilde f} \tilde \omega^n.
\end{equation}

Fix $\delta_2\in (0,1)\setminus \Gamma$ and $\delta_2\leq \frac{\delta_1}{4}$. Here $\Gamma$ is the set in Proposition \ref{eq--right inverse}. Fix $k\geq 2n+2$ and $\alpha \in (0,1)$. Let us define 
	\begin{equation*}
		\mathcal B=\{u\in C^{k+2,\alpha}_{2+\delta_2}(B_{r_0}): \|u\|_{C^{k+2,\alpha}_{2+\delta_2}(B_{r_0})} \leq \epsilon_0\}
	\end{equation*}where $\epsilon_0$ is chosen to be sufficiently small such that $\tilde \omega+\ii \partial\pp u$ is uniformly equivalent to $\tilde\omega_0$ and $r_0$ is to determined later. Let us consider the operator
	\begin{equation*}
		\begin{aligned}
		\mathcal F:& \mathcal B\longrightarrow  C^{k,\alpha}_{\delta_2}(B_{r_0}) \\
 & u \longrightarrow \log \frac{(\tilde\omega+\ii \partial \pp u)^n}{e^{\tilde f} \tilde \omega^n}.
\end{aligned}
	\end{equation*}Then we have 
	\begin{equation}
	\begin{aligned}
		\|\mathcal F(0)\|_{C^{k,\alpha}_{\delta_2}(B_{r_0})}&\leq r_0^{\frac{\delta_1}{4}}\|\tilde f\|_{C^{k,\alpha}_{\frac{\delta_1}{2}}(B_{r_0})}\leq C_0r_0^{\frac{\delta_1}{4}},\\
		\|Q(u)-Q(v)\|_{C^{k,\alpha}_{\delta_2}(B_{r_0})}&\leq C \left(\|u\|_{C^{k+2,\alpha}_{2}(B_{r_0})}+\|v\|_{C^{k+2,\alpha}_{2}(B_{r_0})}\right)\|u-v\|_{C^{k+2,\alpha}_{2+\delta_2}(B_{r_0})}\\
		& \leq C r_0^{\delta_2}\left(\|u\|_{ C^{k+2,\alpha}_{2+\delta_2}(B_{r_0})}+\|v\|_{C^{k+2,\alpha}_{2+\delta_2}(B_{r_0})}\right)\|u-v\|_{C^{k+2,\alpha}_{2+\delta_2}(B_{r_0})}
			\end{aligned}
	\end{equation}
where $Q(u)=\mathcal F(u)-\mathcal F(0)-d_0\mathcal F(u)$.

By the estimate \eqref{eq: needed for Banach spaces}, we can apply Proposition \ref{eq--right inverse}. Let $\mathcal T$ denote the right inverse of $\Delta_{\tilde g}$. Choosing $r_0$ small, we can apply the Banach fixed point theorem to the operator  $N(u):=\mathcal  T(-\mathcal F(0)-Q(u))$.
Therefore we obtain a solution $\varphi \in C_{2+\delta_2}^{k+2,\alpha}(B_{r_0})$. 
\end{proof}

Let $\bar \omega:=\tilde \omega+\ii\partial\pp \varphi$ and $\bar g=\bar\omega(\cdot, J_\ca \cdot)$. 
Although in our setting, $(\bar \omega, J_\ca, \bar g)$ may not be the pointed Gromov-Hausdorff limit of smooth polarized K\"ahler-Einstein metrics, we still have the following:
 \begin{itemize}
 	\item  the regular part is geodesically convex since the space has only an isolated singularity and the metric is asymptotic to a cone with smooth link near this singularity;
 	\item  Donaldson-Sun's 2-step degeneration theory still applies as the tangent cones of $(\bar \omega, J_\ca, \bar g)$ still exists by construction and $o$ has arbitrary small neighborhood which are Stein so that we can use H\"ormander $L^2$-method. 
 \end{itemize} With these two properties, the argument in \cite{CSz} still works for the Calabi-Yau metric $(\bar \omega, J_\ca, \bar g)$ and therefore we get the following polynomial closeness result.

\begin{theorem}[\cite{CSz}]\label{polynomail close of the intermediate calabi-Yau metric}
	There exists a constant $\alpha >0$ and a biholomorphism $\Psi:(\tilde U_o,J_\ca )\rightarrow (U_o, J_\ca)$ such that for all $k\geq 0$, as $r\rightarrow 0$, we have 
	\begin{equation*}
		|\nabla_{g_\ca }^k(\Psi^*g-g_{\ca})|=O(r^{\alpha-k}).
	\end{equation*}
\end{theorem}

Then the polynomial closeness between $(\omega_Z, J_Z, g_Z)$ and $(\omega_\ca, J_\ca, g_\ca)$ follows from \eqref{eq--estimate for tilde}, \eqref{eq: needed for Banach spaces}, Propsition \ref{prop--solve ma to get calabi-yau} and Theorem \ref{polynomail close of the intermediate calabi-Yau metric}.

\section{Examples}\label{sec-examples}
 We construct examples of singular K\"ahler-Einstein metrics, which have a unique tangent cone with smooth link, yet are not polynomially close to their tangent cones. In \cite{GChen}, a Riemannian metric with $G_2$ holonomy, in particular Ricci-flat, and exhibiting only a logarithmic rate of convergence to the cone is constructed. A key feature of our examples is that they can also be realized as Gromov-Hausdorff limits of non-collapsing Calabi-Yau metrics, and therefore in particular showing that the result of Colding-Minicozzi \cite{CM} is sharp in a certain sense. This was first pointed out in \cite{HS}.

Let $D$ be an $(n-1)$-dimensional Fano manifold and $\Omega$ be a nowhere vanishing holomorphic volume form $\Omega$ on $K_D$, the canonical bundle of $D$. Let $h_D$ deonte a Hermitian metric with negative curvature form, $\xi\in K_D$, then $\psi=\|\xi||^2_{h_D}$ is a plurisubharmonic function on $K_D$ and it is strictly plurisubharmonic on $K_D\setminus D$. Then for any $\delta>0$, $\{\psi<\delta\}$ is a domain containing $D$, with a strictly pseudoconvex boundary. Therefore there is a holomorphic map $\pi:(K_D,D)\rightarrow (K_D^{\times},o)$ which blows-down the zero section. We may assume $K_D^{\times}$ is holomorphicallly embedded in $\mathbb C^N$. Let $\omega_{0}$ denote the restriction of the Euclidean metric to $K_D^{\times}$. Using $\pi$, we identify it with an semi-positive form on $K_D$.

Fix a small $\delta>0$ such that $\mathcal U=\{\psi< \delta\}$ has smooth strictly pseudoconvex boundary. Fix a K\"ahler form $\eta$ on $K_D$.  Then for $0<\epsilon< 1$, we consider the following Monge-Ampere equations with Dirichlet boundary condition:
 \begin{equation}\label{eq-cma}
 	\left\{\begin{array}{l}
\left(\omega_0+\epsilon \eta+\sqrt{-1} \partial \bar{\partial} \varphi_\epsilon\right)^n=\ii^{n^2} \Omega\wedge\ols{
\Omega} \\
\left.\varphi_\epsilon\right|_{\partial \mathcal{U}}=0.
\end{array}\right.
 \end{equation}
 
The following result is proved in \cite{Bouc,Fu,GGZ23}, following the fundamental work \cite{CKNS,Kol}. For readers' convenience, we give a sketch of the proof and refer to \cite{Bouc,Fu,GGZ23} for details. Let $\omega_{\epsilon}$ denote $\omega_0+\epsilon \eta+\sqrt{-1} \partial \bar{\partial} \varphi_\epsilon$ when a solution of \eqref{eq-cma} exists.
\begin{proposition}[\cite{Bouc,Fu,GGZ23}]
	For any $\epsilon\in (0,1)$, the equations \eqref{eq-cma} admits a unique smooth solution $\varphi_{\epsilon}$.  There exist  constants $C$ independent of $\epsilon$ such that 
	\begin{equation}
	\left\|\varphi_{\epsilon}\right\|_{L^{\infty}}\leq C\text{ and }
	C^{-1}\omega_0\leq \omega_\epsilon\leq C\frac{\ii^{n^2}\Omega\wedge\ols{\Omega}}{\omega_0^n}\omega_0.
			\end{equation}
			Moreover, for any $k\geq 0 $ and compact subset $K\subset \subset \ols{ \mathcal U}\setminus D$, there exists constants $C_{K,k}$ such that 
			\begin{equation}\label{eq-local higher estimate}
	\|\nabla_{\omega_0}^k\varphi_{\epsilon}\|_{L^{\infty}(K)}\leq C_{K,k}.
		\end{equation}
\end{proposition}

\begin{proof} The uniqueness of the solution follows from the maximal principle for complex Monge-Ampere equations.
	 Note that $A(\psi-\delta)$ is subsolution of \eqref{eq-cma} for $A$ sufficiently large, depending on $\epsilon$. Therefore by \cite[Theorem A]{Bouc}, we know that \eqref{eq-cma} admits a smooth solution for any $\epsilon\in (0,1)$.
	
	  The uniform $L^{\infty}$ bound for $\varphi_{\epsilon}$ is proved in \cite{Fu,GGZ23}.
	   The comparison between $\omega_0$ and $\omega_{\epsilon}$ can be derived from the Chern-Lu's inequality and the uniform $L^{\infty}$ bound for $\varphi_{\epsilon}$, since $\omega_0$ is the restriction of the Euclidean metric to $U$ and hence has a uniform upper bound on the the bisectional curvature. See \cite{DFS, Fu} and references therein for the uniform higher order estimate \eqref{eq-local higher estimate}.
	   \end{proof}

As we are going to take a pointed Gromov-Hausdorff limits of $\omega_{\epsilon}$, we fix a point $p\in D$ and establish the following uniform non-collapsing result. 
\begin{lemma}\label{lemma--diameter bound}
	There exists constants $C$ independent of $\epsilon$ such that 
	\begin{equation}\label{eq-diamter bound}
		\sup_{x\in \mathcal U}d_{\omega_{\epsilon}}(p,x)\leq C. 
	\end{equation} Therefore there exists a $\kappa>0$ such that for every $r <\min\{1, \operatorname{dist}(x,\partial \mathcal U)\}$, we have 
	\begin{equation}
	\operatorname{Vol}(B(x,r),\omega_\epsilon)\geq \kappa r^{2n}.
\end{equation}
\end{lemma}

\begin{proof}
Suppose, on the contrary, that
 \begin{equation*}
 	\sup_{x\in \mathcal U}d_{\omega_{\epsilon}}(p,x)\rightarrow \infty \text{ as } \epsilon\rightarrow 0.
 \end{equation*}
 As we have uniform estimate for $\omega_{\epsilon}$ away from $D$, we know that $d_{\omega_\epsilon}(p, \partial \mathcal U)\rightarrow \infty$ as $\epsilon\rightarrow 0$. Note that we have  $\operatorname{Ric}(\omega_{\epsilon})=0$ and 
	$\int_{\mathcal U} \omega_{\epsilon}^n=\int_{\mathcal U}\ii^{n^2} \Omega\wedge\ols{
\Omega}$ has a uniform upper bound. This contradicts with a theorem of Yau \cite[Theorem I.4.1]{SY}. The second statement follows from the Bishop-Gromov volume comparison.
\end{proof}

Passing to a subsequence, we may assume $\omega_{\epsilon}$ converges in $C^{\infty}_{loc}(\olsi{\mathcal U}\setminus D)$ to a K\"ahler form $\omega$.
Moreover we may assume $\left(\ols{\mathcal U},p,\omega_{\epsilon}\right)$ converges in the pointed Gromov-Hausdorff sense to $(Z,p,d_Z)$. Then a similar argument as in \cite[Section 5]{SZ2023} implies the following. See also \cite{Song}.

\begin{proposition}
 $Z$ coincides with the metric completion of $(\olsi{\mathcal U}\setminus D,\omega)$ and naturally homeomorphic to $\pi(\ols{\mathcal U})$.
\end{proposition} In particular, we can find arbitrary small neighborhood of $p$ with strictly pseudoconvex boundary, H\"omander $L^2$ method applies and therefore one can extend Donaldson-Sun's 2-step degeneration theory to $(Z,p,d_Z)$. Then we get a $K$-semistable cone $W$ and its unique tangent cone $\ca$ at $p$.  Both $W$ and $\ca$ depends only on the local algebraic structure of $(K_D^{\times},o)$ by \cite{LX,LWX}. In particular, if we choose $D$ to be a smooth Fano manifold which does not admit a K\"ahler-Einstein metric, but admit a degeneration to a smooth Fano manifold $D'$ that admits a K\"ahler-Einstein metric. Such examples exist \cite{Don2007note,Tian1997}. Then we know that $$W=K_D^\times \neq  K_{D'}^\times=\ca.$$ More precisely, as $D$ degenerates to $D'$, $K_D^{\times}$ admits an equivariant degeneration to $K_{D'}^{\times}$, which admits a Calabi-Yau cone metric. Therefore $K_D^{\times}$ with the standard Reeb vector field is K-semistable. Moreover as $D$ does not admit K\"ahler-Einstein metrics, $K_D$ (with the standard Reeb vector field) is not K-polystable.  Therefore the result in Section \ref{polynomial convergence imply} in this paper implies that the singular Calabi-Yau metric near $p$ is not polynomially close to its tangent cone $\ca$.  

\section{Local non-positive K\"ahler-Einstein metrics}\label{Section--local}
Let $(X,p)$ be a germ of an isolated log terminal algebraic singularity. There is a well developed theory for local stability of (Kawamata) log terminal singularities; see \cite{LL,LX,LWX,Zhuang2023} and the references therein. Roughly speaking, this local stability theory says that every (Kawamata) log terminal  singularity has a two-step  degeneration to a K-polystable Fano cone.
The first degeneration is canonical and is induced by the unique (up to scaling) valuation that minimizes the normalized volume. Through this degeneration, we get a K-semistable Fano cone $(W,\xi_W)$. Secondly $(W,\xi_{W})$ admits an equivariant degeneration to a K-polystable Fano cone $(\ca,\xi)$. In the following, we will call this process as the 2-step stable degeneration of a log terminal (algebraic) singularity.

In this paper, we consider the special case that $W$ is isomorphic to $\ca$ as Fano cones. For simplicity, we will denote this by $\ca=W$. We further assume that $\ca$ has only an isolated singularity.

We have the following setup as similar to that Section \ref{sec--ds theory}.  $(\ca,\xi)$ admits a Calabi-Yau cone structure, that is there exists a K\"ahler Ricci-flat cone metric $\omega_{\ca}$ on $\ca$ with $\xi=J_{\ca}(r\partial _r)$. There exist embeddings $(\ca,o)\subset (\mathbb C^N,0)$ and $(X,p)\subset (\mathbb C^N,0)$ such that under the embedding, $\xi=J_{\ca}(r\partial _r)$ extends to a linear vector field on $\C^N$ of the form $Re(\sqrt{-1}\sum_i d_i z_i\partial_{z_i})$, which we also denote by $\xi$, where $d_i\in \mathbb R_{>0}$ for all $i$. Moreover the weighted tangent cone of $(X,p)$ with respect to $\xi$ coincides with $\ca$.

Assume we have a singular non-positive K\"ahler-Einstein metric with bounded potential on the germ $(X,p)$.  This means that there exists neighborhood $U$ of $p\in X$ and a function $$\varphi_{KE}\in C^{\infty}(U\setminus \{p\})\cap L^{\infty}(U)$$ such that the K\"ahler form $\omega_{KE}=\ii\partial\pp \varphi_{KE}$ defined on $U\setminus \{p\}$ satisfies 
\begin{equation}\label{eq--initial Kahler-Eisntien metric}
	\operatorname{Ric}(\omega_{KE})=\lambda \omega_{KE}
\end{equation}for some constant $\lambda\leq 0$. By rescaling we may assume $\lambda \in \{-1,0\}$. Since we are only caring about the behavior of $\omega_{KE}$ near $p$, in the following, we will shrink $U$ if necessary without explicit mentioning this.
Let us repeat the statement of Theorem \ref{Theorem--local} here.
\begin{theorem}
Suppose $\ca=W$ and it has only an isolated singularity, then every singular non-positive K\"ahler-Einstein metric with bounded potential on the germ $(X,p)$ is conical at $p$.
\end{theorem}

 The same proof as in Section \ref{polynomial convergence imply} implies that the isomorphism between $\ca$ and $W$ as Fano cones is also a necessary condition for the existence of a conical singular K\"ahler-Einstein metric. To show its sufficiency, we will closely follow Hein-Sun's continuity argument. For the openness, we need to deal with the issue that $(X,p)$ is not biholomorphic to $\ca$. For the closeness, we need Theorem \ref{thm-main theorem} in this paper and results \cite{CW, LX, LWX}, which allow us to remove the smoothability assumption in Hein-Sun's work.

We start with the following result, which gives the initial metric we start with.

\begin{lemma}\label{lemma-initial metric}
	On $U\setminus \{p\}$, there exists a K\"ahler form $\omega_0=\ii\partial\pp \varphi_0$, which is Ricci-flat near $p$ and polynomially close to the Calabi-Yau cone $\ca$ and 
\begin{equation*}
	\varphi_0|_{\partial U}=\varphi_{KE}|_{\partial U}. \end{equation*}	
	
\end{lemma}
\begin{proof}
	 Using the diffeomorphism $\Phi$ obtained using Lemma \ref{lemma: complex structure close} to pull back K\"ahler potantials and solving Monge-Amp\`ere equations as we did in Section \ref{imply polynomial convergence} , we can prove that in a neighborhood of $p$, there exists a conical Calabi-Yau metric $\tilde \omega_0=\ii\partial \pp \tilde\varphi_0.$ As before, we view $r$ as a smooth function in a neighborhood of $p$ via the diffeomorphism. Fix a small positive constant $\delta_1$ such that both $\omega_{KE}$ and $\tilde \omega_0$ are well defined in a neighborhood of $\{r\leq \delta_1\}$ and $r^2$ is strictly plurisubharmonic on it. Then we are going to use a scaling and cut-off argument similar to that in \cite{ArezzoSpotti} to obtain the desired K\"ahler form $\omega_0$. Without loss of generality, we may assume $ \min_{\{r=\delta_1\}} \varphi_{KE}=10$. As $\varphi_{KE}$ is bounded, then we can choose $A$ sufficiently large such that 
	 $\varphi_1:=\varphi_{KE}+A(r^2-\delta_1^2)\leq 0$ on the region $\{r\leq \frac{\delta_1}{2}\}$. Then choose $\epsilon_1\ll \delta_1$ such that $\varphi_1\geq 2$ on $\{\delta_1-\epsilon_1\leq r\leq \delta_1\}$. Then fix two cut-off function $\chi_1$ and $\chi_2$ which are smooth non-negative function on $\mathbb R$ such that \begin{itemize}
	 	\item $\chi_1(t)$ equals a constant for $t\leq \frac{1}{2}$ and $\chi_1(t)\equiv t$ for $t\geq \frac{3}{2}$ and $\chi_1',\chi_1''\geq 0$
	 	\item $\chi_2(t)\equiv 1$ for $t\leq \delta_1-\epsilon_1$ and $\chi_2(t)\equiv 0$ for $t\geq \delta_1-\frac{\epsilon_1}{2}$.
	 \end{itemize} Then one can directly show that for $\epsilon_2$ sufficiently small, 
	 \begin{equation*}
	 	\omega_0:=\ii\partial\pp(\epsilon_2 \chi_2(r)\tilde\varphi_0+\chi_1(\varphi_1)).
	 \end{equation*}is a K\"ahler form and  satisfies the required properties when we choose $U=\{r\leq \delta_1\}$. 
\end{proof}


Since $\omega_{KE}=\ii\partial\pp \varphi_{KE}$ is K\"ahler-Einstein \eqref{eq--initial Kahler-Eisntien metric} and $\omega_0$ is Ricci-flat near $p$, there exists a smooth function $f_1$, which is pluriharmonic in a neighborhood of $p$ such that 
\begin{equation}
	(\ii\partial\pp \varphi_{KE})^n=e^{-\lambda \varphi_{KE}+ f_1}\omega_0^n,
\end{equation}where $\lambda\in \{-1,0\}$.
Then consider the following continuous path: we are looking for functions $\varphi_t\in C^{\infty}(U\setminus p)\cap L^{\infty}(U)$ satisfying
 \begin{equation}\label{eq--continuity path}
 	\begin{cases}
		(\ii\partial\pp \varphi_t)^n=\exp(-\lambda t(\varphi_t+f_1))\omega_0^n  \\
		\varphi_t|_{\partial U}=\varphi_{KE}.
\end{cases}
 \end{equation}
 
 Let $\omega_t=\ii\partial\pp\varphi_t$. Let $I$ be the subset of $[0,1]$ such that $\omega_t$ is polynomially close to $\ca$ near $p$. Then by Lemma \ref{lemma-initial metric}, we have $0\in I$. We are going to show that $I$ is both open and closed, and therefore $I=[0,1]$. Assuming this, then by the uniqueness result \cite[Theorem 1.4]{GGZ}, we obtain $\varphi_1=\varphi_{KE}$ and therefore $\omega_{KE}$ is polynomially close to $\ca$ near $p$.

\subsection{Openness}\label{sec--openness} We first show that the set 
$I$ is open. This is analogous to \cite[Theorem 2.19]{HS}.
\begin{theorem}
	Let $\omega_{T}=\ii\partial\pp\varphi_{T}$ be a solution to \eqref{eq--continuity path} for some $T\in [0,1]$. Suppose $\omega_{T}$ is polynomially close to $\ca$, then there exists  $\tau>0$ such that for all $t\in (T-\tau,T+\tau )$, \eqref{eq--continuity path} admits a solution $\omega_{t}=\ii\partial\pp\varphi_{t}$, which is also polynomially close to $\ca$.
\end{theorem}

\begin{proof}
	By assumption, there exists a $\delta>0$ and a diffeomorphism $\Phi$ from a neighborhood of $p$ in $U$ to a neighnorhood of $o\in \ca$ such that for all $k\geq 0$,
	 \begin{equation}\label{eq--conical at T}
  	|\nabla^k_{g_{\ca}}(\Phi^*J_X-J_{\ca})|+|\nabla^k_{g_{\ca}}(\Phi^*\omega_{T}-\omega_{\ca})|+|\nabla^k_{g_{\ca}}(\Phi^*g_{T}-g_{\ca})|=O(r^{\delta-k}).
  \end{equation}  
  
  Fix $\alpha \in (0,1)$ and a positive number $\beta < \min\{\delta,\mu_1^+,1\}$ such that $(2,2+\beta)\cap \{\mu_i^+\}=\emptyset$, where $\mu_i^+$ are given in \eqref{eq--exceptional roots}. Let $C_{2+\beta}^{2,\alpha}$ denote the weighted Banach space defined using the cone metric $(\Phi^{-1})^*g_\ca$. More precisely, this means the following. We fix a neighborhood $V\subset U$ of $p$, where $\Phi$ is defined and fix a cut-off function $\chi$ on $U$ which equals 1 in a neighborhood of $p$ and has support in $V$. Then a function $f\in C_{2+\beta}^{k+2,\alpha}$ means $\Phi^*(\chi f)\in C_{2+\beta}^{k+2,\alpha}$ which is defined by \eqref{eq--definition of banach space} using $g_\ca$ and $(1-\chi)f \in C^{2,\alpha}$ with respect to a fixed smooth metric on $U$. Let $C_{2+\beta,0}^{2,\alpha}$ denote the Banach subspace of $C_{2+\beta}^{k+2,\alpha}$, consisting of functions which have zero boundary value.
  
  
 Let $\mathcal{PH}$ denote the vector space of homogeneous $J_\ca$-pluriharmonic real-valued functions with growth rate in $(0,2]$. It is direct to show that these $J_\ca$-pluriharmonic functions are real part of $J_\ca$-holomoprhic functions \cite[Lemma 2.13]{HS}. As $J_X$ is polynomially close to $J_\ca$, using H\"ormander $L^2$ method as we did in Section \ref{polynomial convergence imply}, we can construct $J_X$-pluriharmonic functions with growth rate in $(0,2]$ that converges to homogeneous $J_\ca$-pluriharmonic functions. Let $\mathcal P$ denote a vector space spanned by $\dim \mathcal{PH}$-dimensional $J_X$-pluriharmonic functions with growth rate in $(0,2]$. $\mathcal P$ is not unique and any choice of such a space works for the following argument.    Let $\mathcal H$ denote the space of all
$\xi$-invariant 2-homogeneous $\Delta_{g_\ca}$-harmonic real-valued functions on $\ca$. Transverse automorphisms of the cone $\ca$, i.e., automorphisms of $\ca$ that preserve the Reeb vector field $\xi=J_\ca(r\partial r)$, form a complex Lie group. Let $G$ denote the connected component of this Lie group. Consider spaces of functions
  \begin{equation}
  	  \mathcal X=C_{2+\beta,0}^{2,\alpha}\oplus \chi \mathcal P\oplus \chi\cdot \frac{1}{2}\left(r^2 \circ \phi-r^2\right) \quad \text{and} \quad \mathcal Y=C_{\beta}^{0,\alpha},
  \end{equation}
 where $\phi\in G$ and using the diffeomorphism $\Phi$ we view $\chi\cdot \frac{1}{2}\left(r^2 \circ \phi-r^2\right)$ as functions in a neighborhood of $p$. It is shown in the proof of \cite[Theorem]{HS} that $\mathcal X$ admits a $C^1$ Banach manifold structure, and moreover, we have 
\begin{equation}
	T_0\mathcal X=C_{2+\beta,0}^{2,\alpha}\oplus \chi \mathcal P\oplus \chi\cdot \frac{1}{2}r^2 \oplus \chi \mathcal H.
\end{equation}

  Then there exists a neighborhood $\mathcal U$ of $(0,T)$ in $\mathcal X\times \mathbb R$ and a neighborhood $\mathcal V$ of $0$ in  $\mathcal Y$ such that the following map $\mathcal M:\mathcal U\rightarrow \mathcal V$ is well defined and is $C^1$: 
  \begin{equation}
  	\mathcal M(\varphi):=\log \frac{(\omega_T+\ii\partial\pp \varphi)}{\omega_0^n}+\lambda t(\varphi_T+\varphi+f_1).
  \end{equation}Indeed, we just need to show $f_1\in C^{0,\alpha}_\beta$. Since $f_1$ is $J_X$-pluriharmonic, in particular it is $\omega_T$-harmonic. Then by the gradient estimate, we know $|\nabla_{\omega_T}f_1|\leq C$ and hence $f_1\in C^{0,\alpha}_\beta$ by the polynomial closeness \eqref{eq--conical at T}.
  The linearization of $\mathcal M$ at $(0,T)$ along the first component is given by 
  \begin{equation*}
  	\Delta_{\omega_T}+\lambda T\operatorname{Id}: C_{2+\beta,0}^{2,\alpha}\oplus \chi \mathcal P\oplus \chi\cdot \frac{1}{2}r^2 \oplus \chi \mathcal H\rightarrow C_{\beta}^{0,\alpha}
  \end{equation*}This is invertible as shown in Appendix 1. The implicit function theorem implies that \eqref{eq--continuity path} has a solution   
  $\omega_t=\omega_T+\ii\partial\pp \psi_{t}$ with $\psi_t \in \mathcal{U}$ for some $\tau>0$ and all $t \in(T-\tau,T+\tau)$.

It remains to prove that $\omega_t$ is polynomially close to $\ca$ near $p$. By construction $\psi_t$ has three terms, that is $\psi_t=u_1+\chi u_2+\frac{1}{2}(r^2\circ \phi-r^2)$, where $\phi\in G$ is a transverse automorphisms of $\ca$. Then we are going to show that $\Phi\circ \phi^{-1}$ gives the desired diffeomorphism. Note that we have the following:
\begin{equation}\label{eq--want to show complex close}
	\nabla^k_{g_{\ca}}((\Phi\circ \phi^{-1})^*J_X-J_{\ca})=(\phi^{-1})^*(\nabla_{\phi^*g_\ca}^k(\Phi^*J_X-J_\ca)),
\end{equation}
  \begin{equation}
  	\nabla^k_{g_{\ca}}((\Phi\circ \phi^{-1})^*\omega_t-\omega_{\ca})|=(\phi^{-1})^*(\nabla_{\phi^*g_\ca}^k(\Phi^*\omega_t-\phi^*\omega_{\ca})).
  \end{equation} 
  As $g_{\ca}$ and $\phi^*g_{\ca}$ are Riemannian cone metrics with the same scaling vector field, the weighted analysis with respect to the metric $g_\ca$ is comparable to that with respect to $\phi^*g_\ca$. This can be seen for example by writing $\phi^*g=g(\theta\cdot,\cdot)$. Then we have $\mathcal L_{r\partial r}\theta=0$ and as a consequence we have $|\nabla_{g_\ca}^j\theta|\sim r^{-j}$. Then \eqref{eq--want to show complex close} follows from \eqref{eq--conical at T} and a direct computation.  
 Since $u_1\in C_{2+\beta,0}^{2,\alpha}$ and $u_2$ is $J_X$-plurisubharmonic, and \eqref{eq--conical at T}, we know that 
 \begin{equation}
 	 \Phi^*\omega_t-\phi^*\omega_{\ca}=(\Phi^*\omega_T-\omega_\ca)+d \eta
 	  	 \end{equation}
 for some 1-form $\eta$ satisfying 
 \begin{equation}\label{eq--estimate for eta}
 	|\nabla^j_{g_\ca}\eta|\leq C \left(|\nabla_{g_\ca}^j(\Phi^*(J_Xd u_1))|+|\nabla^j_{g_\ca}(\Phi^*J_X-J_\ca)(d(r^2\circ\phi-r^2))\right).
 \end{equation}
Then using scaled Schauder estimates, we can get higher order estimate for $u_1$ and hence for $\Phi^*\omega_t-\phi^*\omega_{\ca}$ using \eqref{eq--conical at T} and \eqref{eq--estimate for eta}.
\end{proof}

\subsection{Closeness} In this section, we are going to show the set $I$, which consists of $t$ such that $\omega_t$ is polynomially close to $\ca$ near $p$, is closed, i.e., if $t_i\in I$ and $t_i\rightarrow T$, then $T\in I$. Firstly, we recall the following uniform estimate proved in \cite{DFS,Fu,GGZ23}.

\begin{proposition}
	There exists a constant $C_0$ independent of $t$ such that 
	$\left\|\varphi_{t}\right\|_{L^{\infty}}\leq C_0$.
			Moreover, for any $k\geq 0 $ and compact subset $K\subset \subset \ols{U}\setminus \{p\}$, there exists constants $C_{K,k}$ such that 
			\begin{equation}\label{eq--uniform higher estimate}
	\|\nabla_{\omega_0}^k\varphi_{t}\|_{L^{\infty}(K)}\leq C_{K,k}.
		\end{equation}
\end{proposition}

%
%

Next we want to obtain a uniform diameter upper bound for $\omega_t$.
As  we have the uniform estimate \eqref{eq--uniform higher estimate} away from $p$, it is sufficient to show that $\operatorname{diam}(V,\omega_t)$ has a uniform upper bound, where $V$ is a small neighborhood of $p$. Notice that we can choose $V$ small such that  $\omega_t$ is K\"aher-Einstein on $V$. In the Calabi-Yau setting, i.e. $\lambda=0$, this diameter upper bound can be proved in the same way as Lemma \ref{lemma--diameter bound}. For negative K\"ahler-Einstein metrics, the uniform diameter bound can be proved following \cite{GPSS23}. For readers' convenience, we give some details in Appendix 2.
\begin{proposition}
 	There exists a constant independent of $t$ such that 
 	\begin{equation}
 		\operatorname{diam}(U,\omega_t)\leq C.
 	\end{equation}
 \end{proposition}

As illustrated in \cite{CW}, the Bishop-Gromov volume comparison still holds for the metrics defined by $\omega_t$, therefore this uniform diameter upper bound implies volume non-collapsing, that is there exists a $\kappa>0$ such that for every $r <\min\{1, \operatorname{dist}(x,\partial U)\}$, we have 
\begin{equation}
	\operatorname{Vol}(B(x,r),\omega_t)\geq \kappa r^{2n}.
\end{equation}
Then we can use Chen-Wang's compactness \cite{CW}, which generalize results in \cite{ChCo1,CCT} to metric space swith mild singularities. Note that we have uniform estimate of $\omega_{t}$ away from $p$, therefore the boundary of $U$ will not cause extra issues when we discuss the convergence of $\omega_t$.
Although in \cite{CW}, they only state the results for Calabi-Yau metrics (as they are mainly interested in complete non-compact metrics spaces there), it is direct to check that the statements also hold of K\"ahler-Einstein metrics. We refer to \cite{CW} for the precise meaning of the convergence in the following. We only remark here that this ${\hat{C}^{\infty}}$ implies pointed  Gromov-Hausdorff convergence.
\begin{theorem}[\cite{CW}]
	For each sequence $\omega_{t_i}$ by taking subsequence if necessary, we have
	\begin{equation}
		(U,p,\omega_{t_i})\xrightarrow{\hat{C}^{\infty}}(U_\infty, p_\infty, d_\infty).
	\end{equation}Moreover we have the following.
		\begin{itemize}
		\item $(U_\infty, d_{\infty})$ is a length space and has a regular–-singular decomposition $U_\infty=\mathcal R\cup\mathcal S$ such that there exists a K\"ahler structure $(g_\infty,J_\infty)$ on $\mathcal R$ and $\mathcal S$ is closed subset and has Hausdorff codimension at least  4.
		\item The distance structure induced by $g$ is the same as the restriction of $d_\infty$ and $\mathcal R$ is geodesically convex.
\item Every tangent space of $x \in \mathcal S$ is a metric cone.
	\end{itemize}
	\end{theorem}

Since we have uniform estimate of $\omega_{t}$ away from $p$, passing to a subsequence, we also know that $\omega_{t_i}=\ii\partial\pp \varphi_{t_i}\rightarrow \ii\partial\pp \varphi_{T}$ in $C_{loc}^{\infty}(U\setminus p)$. Let $g_T$ denote the Riemannian metric corresponds to the K\"ahler form $\ii\partial\pp \varphi_{T}$.
	Then using the argument in \cite{DS2} \cite[Section 5]{SZ2023}, we obtain the following.
\begin{proposition}  
 $(U_\infty, g_{\infty}, J_\infty)$ coincides with the metric completion of $(U\setminus p, g_T, J_X)$ near $p$, which is naturally homeomorphic to $U$.
\end{proposition}

Then Donaldson-Sun's 2 step degeneration theory still holds in this setting, as we can use H\"ormander $L^2$-method.  Let $\mathcal C'$ denote the tangent cone of $(g_\infty,J_\infty)$ at $p$. Then by \cite{LWX}, this $\mathcal C'$ has to be isomorphic to $\mathcal C$, which has an isolated singularity. Then the argument in Section \ref{imply polynomial convergence} can also be applied to $(g_\infty, J_\infty)=(\omega_T, J_X)$ to get polynomial closeness and therefore $T\in I$.  

\subsection{Positive singular K\"ahler-Einstein metrics}\label{Section--positive ricci}

This continuity method does not work for singular K\"ahler-Einstein metrics with positive Ricci curvature, due to the lack of the openness and uniqueness; see \cite[Section 5]{Guedj2013} for a related discussion. See \cite{GT23} for using variational method to study the existence of singular K\"ahler-Einstein metrics in a neighborhood of log terminal singularities. 

It is very likely that in the global setting we can avoid such an issue. Let $\omega_{KE}$ denote a singular K\"ahler-Einstein metric on a $\mathbb Q$-Fano variety $X$, which has only isolated singularities and $\ca=W$ for every point in $X^{sing}$. 
 To show that $\omega_{KE}$ is conical, one may need two steps.
 
  The first step is to construct a K\"ahler metric on $X$ with non-negative Ricci curvature and polynomially close to Calabi-Yau cones near singular points of $X$. Let $\theta_0 \in c_1(K_X^{-1})$ be a smooth K\"ahler form. Solving a complex Monge-Amp\`ere locally as we did in Section \ref{imply polynomial convergence} and using the cut-off and glueing trick in \cite{ArezzoSpotti}, we obtain a K\"ahler form $\alpha_0=\theta_0+\ii\partial\pp \psi_0$, which is Ricci-flat and polynomially close to Calabi-Yau cones near singularities of $X$, and a non-negative (1,1)-form $\eta\in c_1(K_X^{-1})$, which is zero in a neighborhood of $X^{sing}$ and coincide with $\theta_0$ outside a neighborhood of $X^{sing}$. Then we know that there exists a function $f$, which is smooth on $X^{reg}$ and pluriharmonic in a neighborhood of $X^{sing}$ and satisfies
\begin{equation}
	\operatorname{Ric}(\alpha_0)-\eta=\ii\partial\pp f.
\end{equation}
We consider the first continuous path
\begin{equation}\label{eq--first path}
	(\theta_0+\ii\partial\pp \psi_t)^n=c_te^{tf}\alpha_0^n.
\end{equation}where $c_t$ is a normalization constant. Note that along the path \eqref{eq--first path}, the K\"ahler metrics are Ricci-flat in a neighborhood of $X^{sing}$ and one can adapt the argument  in Section \ref{Section--local} to show that every metric in this continuous path, is polynomially close to Calabi-Yau cones near $X^{sing}$. As a output, we get a K\"ahler metric $$ \alpha:={\theta}_0+\ii\partial\pp\psi_1,$$ with $\operatorname{Ric}(\alpha)=\eta$ and polynomially close to Calabi-Yau cones near $X^{sing}$.

 The second continuity path is the classical Aubin's:
  \begin{equation}\label{eq--continuous path in positive case}
	(\alpha+\ii\partial\pp \varphi_t)^n=e^{h-t\varphi_t}\alpha^n. 
\end{equation} Here $h$ is a smooth function on $X^{reg}$, and in a neighborhood of $X^{sing}$ can be decomposed as $h_0+h_1$, where $h_0$ is pluriharmonic and $h_1$ satisfies  $|\nabla^k_{\alpha} h_1|=O(r^{2-k})$, and satisfies
\begin{equation*}
	\operatorname{Ric}(\alpha)- \alpha=\ii\partial\pp h \text{ and } \int_X (e^h-1)\alpha^n=0.
\end{equation*}
We want to show that along this continuous path, every metric is conical. Let $\omega_t$ be a solution of \eqref{eq--continuous path in positive case}, then $\operatorname{Ric}( \omega_t)=t\omega_t+(1-t)\alpha$.  Openness for $t\in (0,1)$ can be proved in a similar vein to Aubin's work \cite{Aubin} and openness at $t=0$ can be proved using a trick of Sz\'ekelyhidi \cite[Section 3.5]{Szekelyhidi2014book}. Darvas \cite{Darvas} proves that the existence of a singular K\"ahelr-Einstein metric implies that the $K$-energy is proper in a certain sense. Given that the functional $J_{\chi}$ is convex along smooth geodesics when $\chi$ is a semispositive form proved by Chen \cite{Chen2000}, and holomorphic vector fields generate smooth geodesics in the space of K\"ahler potentials, one may adapt the argument in \cite[Section 2]{Tian2012} to obtain the uniform $C^0$ bound for $\varphi_t$ and hence higher order estimate away from $X^{sing}$. To show that $\omega_0$ is conical, we need to use another continuous path 
\begin{equation}\label{eq--path for initial metric}
	(\alpha+\ii\partial\pp u_t)^n=c_te^{th}\alpha^n.
\end{equation}
Then one may generalize the compactness and regularity results in \cite{ChCo1,CCT, CW} to non-collapsing K\"ahler metrics with isolated singularities and Ricci curvature lower bound, and then establish the closeness of being conical in the continuous path \eqref{eq--continuous path in positive case} and \eqref{eq--path for initial metric}.

\section*{Appendix 1: Laplacian for conical metrics}

The result in this appendix is standard and we just aim to give a detailed and elementary proof of the result we used in Section \ref{sec--openness}. 
First we recall some notations. Let $(g_\ca,J_\ca)$ denote a Calabi-Yau cone metric and $(g,J)$ denote a conical metric of order $\delta$ in the sense of Definition \ref{def--conical} defined on $B_2$, the ball of radius 2 centered at the vertex of the cone. The weighted H\"older spaces are defined by \eqref{eq--definition of banach space} and $C_{2+\beta,0}^{k+2,\alpha}$ is the Banach subspace of $C_{2+\beta}^{k+2,\alpha}$, consisting of functions which have zero boundary value.

  Let $\mathcal{PH}$ denote the vector space of homogeneous $J_\ca$-pluriharmonic real-valued functions with growth rate in $(0,2]$. It is direct to show that these $J_\ca$-pluriharmonic functions are real part of $J_\ca$-holomoprhic functions \cite[Lemma 2.13]{HS}. As $J$ is polynomially close to $J_\ca$, using H\"ormander $L^2$-method as we did in Section \ref{polynomial convergence imply}, we can construct $J$-pluriharmonic functions with growth rate in $(0,2]$ that converges to homogeneous $J_\ca$-pluriharmonic functions. Let $\mathcal P$ denote a vector space spanned by $\dim \mathcal{PH}$-dimensional $J$-pluriharmonic functions with growth rate in $(0,2]$. $\mathcal P$ is not unique and any choice of such a space works for the following argument. Let $\mathcal H$ denote the space of all
$\xi$-invariant 2-homogeneous $\Delta_{g_\ca}$-harmonic real-valued functions on $\ca$.

Let $0=\lambda_0<\lambda_1 \leqslant \lambda_2 \leqslant \ldots$ denotes the eigenvalues of the Laplacian on the link $L$ of $\ca$ (listed with multiplicity), and $\phi_0, \phi_1, \phi_2, \ldots$ denote an associated orthonormal basis of eigenfunctions. Then $r^{\mu_{i}^{\pm}}\phi_i$ are homogenenous harmonic functions on $\ca$, where
\begin{equation*}
	\mu_i^{ \pm}=-\frac{m-2}{2} \pm \sqrt{\frac{(m-2)^2}{4}+\lambda_i}.
\end{equation*}
Fix $\alpha \in (0,1)$ and a positive number $\beta < \min\{\delta,\mu_1^+,1\}$ such that $(2,2+\beta)\cap \{\mu_i^+\}=\emptyset$. 
 Fix a cut-off function $\chi$ such that $\chi\equiv 1$ near the vertex and $\operatorname{Supp}(\chi)\subset \subset B_1$.
 The main result of this appendix is the following. Note that due to the choice of the parameters and $(g,J)$ being conical, the operator is well-defined.

\begin{proposition}
	For any $t\geq 0$, the following operator is invertible
	\begin{equation*}
  	\Delta_{g}-t\operatorname{Id}: C_{2+\beta,0}^{k+2,\alpha}(B_1)\oplus \chi \mathcal P\oplus \chi\cdot \frac{1}{2}r^2 \oplus \chi \mathcal H\rightarrow C_{\beta}^{k,\alpha}(B_1).
  \end{equation*}\end{proposition}
  
\begin{proof} 
	By the maximal principle, this operator is injective, therefore it is sufficient to show the surjectivity.
	It is enough to show that $\Delta_{g}: C_{2+\beta,0}^{k+2,\alpha}\rightarrow C^{k,\alpha}_\beta$ is a Fredholm operator with index bounded below by $-\operatorname{dim}(\mathcal P)-\operatorname{dim}(\mathcal H)-1$.
	
	 Note that by the Schauder estimate, the operator $\Delta_{g_{\ca}}: C_{2+\beta,0}^{k+2,\alpha}\rightarrow C^{k,\alpha}_\beta$ has closed image. Let $V$ denote the vector space spanned by $r^{\mu_i^+}\phi_i$, $\mu_i^+\in (0,2]$. Then we are going to show that 
	\begin{equation}\label{eq--isomorphism}
		\Delta_{g_{\ca}}: C_{2+\beta,0}^{k+2,\alpha}\oplus\chi \cdot V\rightarrow C^{k,\alpha}_\beta
	\end{equation}
 is an isomorphism. It suffices to show the surjectivity. By Proposition \ref{eq--right inverse}, for any $f\in C^{k,\alpha}_{\beta}$, there exists $\bar u\in  C_{2+\beta}^{k+2,\alpha}$ such that	\begin{equation*}
		\Delta_{g_\ca}\bar u=f, \quad \text { and }\quad \|\bar u\|_{C^{k+2,\alpha}_{2+\beta}(B_1)}\leq C \|f\|_{C^{k,\alpha}_{\beta}(B_1)}.
	\end{equation*}
 Expand $\bar u$ into Fourier series.
\begin{equation}
	\bar u=\sum_{\mu_i^+>0}\bar u_i(r)\phi_i(y).
\end{equation}Consider the following function $u$  
\begin{equation*}
	\sum_{\mu_i^+>2}(\bar u_i(r)-\bar u_i(1)r^{\mu_i^+})\phi_i(y)+\sum_{\mu_i^+\in (0,2]}(\bar u_i(r)-(1-\chi)\cdot \bar u(1)  r^{\mu_i^+})\phi_i(y)+\chi \cdot \sum_{\mu_i^+\in (0,2]}\bar u_i(1)r^{\mu_i^+}\phi_i.
\end{equation*}Then we know $u\in C^{k+2,\alpha}_{2+\beta,0}$ and $\Delta_{g_\ca} u=\Delta_{g_\ca}\bar u=f$. By Theorem \ref{thm--HS-harmonic functions}, we have $$\dim V=\operatorname{dim}(\mathcal P)+\operatorname{dim}(\mathcal H)+1,$$ therefore $\Delta_{g_\ca}$ has index $-\operatorname{dim}(\mathcal P)-\operatorname{dim}(\mathcal H)-1$.

It is proved in \cite[Theorem 6.1]{LockOwen} that $\Delta_{g}$ is Fredholm and has the same index as $\Delta_{g_\ca}$. We give a sketch of the proof here. As $g$ is conical, by writing $g(\cdot,\cdot)=g_\ca(\theta \cdot,\cdot)$ and consider $g_s:=g_\ca(\theta^s \cdot,\cdot)$ for $s\in [0,1]$ we have a smooth family of conical metrics connecting $g_\ca$ and $g$. As the index of Fredhlom operators is a continuous function. it is enough to show that for every conical metric $g$, $\Delta_g: C_{2+\beta,0}^{k+2,\alpha}(B_1)\rightarrow C^{k,\alpha}_\beta(B_1)$ is Fredholm. 

By solving Dirichlet problem on $B_1\setminus B_{\epsilon}$ and taking limits and using Schauder estimate, we know that there exists a bounded operator $P_1:C^{k,\alpha}_\beta\rightarrow C^{k+2,\alpha}_{0,0}$ such that $\Delta_g\circ P_1=\operatorname{Id}$.
For $r\in (0,1)$, let $\rho_r$ denote the cut-off function which equals 1 on $B_1\setminus B_r$ and has support on $B_1\setminus B_{\frac{r}{2}}$. Then consider the operator 
\begin{equation}\Delta':=\Delta_{g_\ca}+(1-\rho_r)(\Delta_g-\Delta_{g_\ca}).
\end{equation} As $g$ is conical, $\Delta'-\Delta_{g_\ca}$ has sufficiently small operator norm  if  $r$ sufficiently small. Fix an $r$ sufficiently small such that $\Delta'$ is Fredholm. Let $P_2$ denote a Fredholm inverse of $\Delta'$, that is both $\operatorname{Id}-\Delta'\circ P_2$ and $\operatorname{Id}-P_2\circ \Delta'$ are compact operators.  We are going to use $P_1$ and $P_2$ to construct a Fredholm inverse of $\Delta_g$, thereby proving  it is Fredholm.

Let $\chi_1$ be a cut-off function, which has support on $B\setminus B_{\frac{r}{4}}$ and equals 1 on $\operatorname{supp}(\rho_{\frac{r}{2}})$. Let $\chi_2$ be a cut-off function, which has support on $B_{\frac{r}{2}}$ and equals 1 on $\operatorname{supp}(1-\rho_{\frac{r}{2}})$. Then one can check the following operator
\begin{equation}
	T(f):=\chi_1 P_1(\rho_{\frac{r}{2}}f)+\chi_2P_2((1-\rho_{\frac{r}{2}})f) 
\end{equation}making
$\operatorname{Id}-\Delta_g\circ T$ and  $\operatorname{Id}-T\circ \Delta_{g}$ compact operators. Let us check this property for $\operatorname{Id}-\Delta_g\circ T$ and the other is similar. First notice that the inclusion operator $C^{l,\alpha}_\beta\rightarrow C^{k,\alpha}_\beta$ is compact for $l>k$. Then 
\begin{equation}
	f-\Delta_g\circ T(f)=f-\chi_1\rho_{\frac{r}{2}}f+\chi_2\Delta_g\circ P_2((1-\rho_{\frac{r}{2}})f)+K(f)
\end{equation}for some compact operator $K$. Notice that on $\operatorname{supp}(\chi_2)$, $\Delta'=\Delta_g$, since $\operatorname{Id}-\Delta'\circ P_2$ is a compact operator, we obtain that 
 \begin{equation}
	f-\Delta_g\circ T(f)=f-\chi_1\rho_{\frac{r}{2}}f+\chi_2(1-\rho_{\frac{r}{2}})f+K'(f)
\end{equation}for some compact operator $K'$.  Notice that $\chi_1\equiv 1$ on $\operatorname{supp}(\rho_{\frac{r}{2}})$ and $\chi_2$ equals 1 on $\operatorname{supp}(1-\rho_{\frac{r}{2}})$, therefore we get $\operatorname{Id}-\Delta_g\circ T=K'$ is a compact operator.
\end{proof}

\section*{Appendix 2: Diameter bounds for K\"ahler manifolds with boundary}
The results in this section follow essentially from \cite{GPSS22,GPSS23}.   For readers' convenience, we give a sketch of the proof and mention some modifications needed for manifolds with boundary. Results here are not optimal, but are sufficient for our purpose. 


We recall the setting considered in Section \ref{Section--local}. That is we have a normal Stein space $(U,p)$ with an isolated log terminal singularity, embedded into $(\mathbb C^N,0)$ with $\partial U$ being strictly pseudoconvex. We have a smooth function $f_1$ which is pluriharmonic in a neighborhood of $p$. For some $\lambda \in  \{-1,0\}$, we have a family of K\"ahler metrics with uniformly bounded potential $\omega_t=\partial\pp \varphi_t$, $t\in [0,1]$, satisfying
 \begin{equation}\label{eq--continuity path in appendix}
 	\begin{cases}
		(\ii\partial\pp \varphi_t)^n=\exp(-\lambda t(\varphi_t+f_1))\omega_0^n \\
		\varphi_t|_{\partial U}=\varphi_{KE}.
\end{cases}
\end{equation}
Moreover we know that each $\omega_t$ is conical at $p$.
The main result in this appendix is the following uniform diameter estimate.

 \begin{proposition}\label{prop--diameter bound}
 	There exists a constant $C$ independent of $t$ such that 
 	\begin{equation}
 		\operatorname{diam}(U,\omega_t)\leq C
 	\end{equation}
 \end{proposition}

In the following, we will omit the index $t$ and just write $\omega_t$ as $\omega$. It is clear from the statement and proof that the estimates in the following results are uniform for $t$ and only depend on the estimates in \eqref{eq--uniform Linfinity estimate}.
Let $\omega_{E}=\ii\partial\pp \varphi_E$ be the restriction of the Euclidean metric on $\mathbb C^N$ and let $e^F=\frac{\omega_0^n}{\omega_E^n}.$ By our assumption, we know that there exists a constant $A$ and $p>1$ such that 
\begin{equation}\label{eq--uniform Linfinity estimate}
\begin{aligned}
	\|\varphi_t\|_{L^{\infty}}+\|f_1\|_{L^{\infty}}&\leq A,\\
\int_Ue^{pF}\omega_E^n&\leq A.
\end{aligned}
\end{equation}


\subsection*{Green functions and Sobolev inequalities} Firstly we recall some basic properties for Green functions on manifolds with boundary \cite[Chapter 4]{Aubin-book}. Let $G_{\omega}(x,y)$ denote the Green function for the Laplacian $\Delta_{\omega}$ on $X$ which is a manifold with boundary. It satisfies the following
\begin{itemize}
	\item [(1)] $G_{\omega}$ is defined and smooth on $\ols X\times \ols X\setminus \{(x,x):x\in X\}$ and $G_{\omega}(x,y)=G_{\omega}(y,x)$.
	\item [(2)] $G_{\omega}(x,y)\geq 0$ and equals zero if and only if $x\in \partial X$ or $y\in \partial X$.
	\item [(3)]For every $u\in C^2(\ols{X})$, we have 
	 \begin{equation}\label{eq--green formula}
	 \begin{aligned}
	 	u(x)&=-\int_X G_{\omega}(x, y) \Delta_{\omega} u(y)dV-\int_{\partial X}\partial _{\vec n}G_{\omega}(x,y)u dS\\
	 	&=\int_{X}\left<\nabla G_{\omega},\nabla u\right>dV-\int_{\partial X}\partial _{\vec{n}}G_{\omega}u dS,
	 \end{aligned}
	\end{equation} where $\vec{n}$ is the unit out-normal vector of the boundary and $dS$ is the volume element  induced by the Riemannian metric on the boundary.
\end{itemize} 
Let $U_{\delta}$ denote an exhaustion of $U\setminus\{p\}$ with smooth boundary. For example, we can choose 
\begin{equation*}
	U_{\delta}=U\cap \{|z|>\delta\}.
\end{equation*}Let $G_{\omega,\delta}(x,y)$ denote the Green function for the Laplacian $\Delta_{\omega}$ on the region $U_{\delta}$.

\begin{lemma}\label{lemma-higher integral bound}
	There exists constants $\epsilon_0=\epsilon_0(n,p)$ and $C_1=C_1(n,p,A,\epsilon_0)$ such that for all $\delta>0$, we have
	\begin{equation}\label{L1+epslion estimate}
		\sup_{x\in U_\delta}\int_{U_\delta}G_{\omega,\delta}(x,\cdot)^{1+\epsilon_0}\omega^n\leq C_1.
	\end{equation}
\end{lemma}

\begin{proof}
	\noindent\textbf{Step 1.} Let $v$ solve the Laplacian equation	
	\begin{equation*}
		\Delta_\omega v=1 \text{ in } U_\delta \quad \text{ and }\quad v=0 \text{ on }\partial U_\delta.
	\end{equation*}
	By the maximal principle, we have $\|v\|_{L^\infty}\leq \frac{2}{n}\|\varphi_t\|_{L^{\infty}}\leq A$. Applying the Green formula  \eqref{eq--green formula} to $v$, we obtain
	 \begin{equation}\label{eq--L1 estimate of green function}
		\sup_{x\in U_{\delta}}\int_{U_{\delta}}G_{\omega,\delta}(x,\cdot)\omega^n\leq A.	\end{equation}	
	\noindent\textbf{Step 2.} Fix a large number $k$ and consider a smooth positive function $H_k$ which is a smoothing of $\min\{G_{\omega,\delta}(x,\cdot ),k\}$. Then by \eqref{eq--L1 estimate of green function} and \eqref{eq--uniform Linfinity estimate}, and H\"older inequality, we obtain there exist constant $\epsilon_0, A_0$ and depends only on $p,n,$ and $A$ such that 
	\begin{equation}\label{eq--lp control}
		\int_U(H^{n\epsilon_0}_ke^F)^{p}\omega_E^n\leq A_0. 
	\end{equation}
	Let $\psi_k$ be the solution of the following complex Monge-Amp\`ere equation
	\begin{equation}
	\begin{cases}
			(\ii\partial\pp \psi_k)^n=H_k^{n\epsilon_0}e^F\omega_E^n \text{ in } U\\
			\psi_k|_{\partial U}=0.
	\end{cases}
	\end{equation} Then by \cite{GGZ23} and \eqref{eq--lp control}, we have a uniform $L^{\infty}$ bound on $\psi_k$. Let $v_k$ solves the equation
	\begin{equation*}
		\Delta_\omega v_k=H_k^{\epsilon_0}\text{ in } U_\delta \quad \text{ and }\quad v=0 \text{ on }\partial U_\delta.
	\end{equation*} By the maximal principle, 
	\begin{equation}
		\|v_k\|_{L^\infty}\leq C \|\psi_k\|_{L^\infty}\leq C.
	\end{equation}Applying the Green formula to $v_k$ and letting $k\rightarrow \infty$, we obtain \eqref{L1+epslion estimate}.
\end{proof}

\begin{proposition}\label{prop--sobolev inequality}
	There exist constants $q=q(n,p)$ and $C_2=C_2(\epsilon_0,n,p,C_1,q)$ such that for $u\in W^{1,2}_0(U_{\delta})$, we have 
	\begin{equation}\label{eq--sobolev inequality}
		\left(\int_{U_\delta}|u|^{2q}\omega^n\right)^{\frac{1}{q}}\leq C\int_{U_\delta}|\nabla u|_{\omega}^2\omega^n.
	\end{equation}
\end{proposition}
\begin{proof}
	It is sufficient to prove this for $u\in C^{\infty}_c(U_\delta)$. Applying \eqref{eq--green formula} to a constant function and then applying the Green formula \eqref{eq--green formula} to $(G_{\omega,\delta}(x,\cdot)+1)^{-\beta}$, for any $\beta>0$, we get 
	\begin{equation}\label{eq--gradient bound}
		\sup _{x \in U_{\delta}} \int_{U_\delta} \frac{\left|\nabla_y G_{\omega,\delta}(x, \cdot)\right|_{\omega}^2}{(G_{\omega,\delta}(x, \cdot)+1)^{1+\beta}} \omega^n\leq \frac{1}{\beta}.
	\end{equation}
	Applying \eqref{eq--green formula} to $u$ and using H\"older inequality and \eqref{eq--gradient bound}, we obtain that for any $\beta,q>0$, we have 
	\begin{equation}
		|u(x)|^{2q}\leq \frac{1}{\beta^{q}}\left(\int_{U_\delta}(G_{\omega,\delta}(x,y)+1)^{1+\beta}|\nabla u(y)|^2\omega^n(y)\right)^{q}.
	\end{equation}
	Integrating over  $U_\delta$ and using Minkowski inequality, one obtains that 
	\begin{equation}\label{eq--almost sobolev}
		\left(\int_{U_\delta}|u|^{2q}\omega^n\right)^{\frac{1}{q}}\leq \frac{1}{\beta} \int_{U_\delta}\left(\int_{U_\delta} (G_{\omega,\delta}(x,y)+1)^{(1+\beta)q}\omega^n(x)\right)^{\frac{1}{q}}|\nabla u|^2\omega^n(y).
	\end{equation}
Let $\epsilon_0$ be the constant obtained in Lemma \ref{lemma-higher integral bound} and choose $\beta=\frac{\epsilon_0}{2}$ and $q=\frac{1+\epsilon_0}{1+\epsilon_0/2}>1$. Then \eqref{eq--sobolev inequality} follows from \eqref{eq--almost sobolev} and Lemma \eqref{lemma-higher integral bound}.
\end{proof}


\subsection*{Diameter bound} Then we are ready to prove Proposition \ref{prop--diameter bound} following \cite{GPSS22,GPSS23}.

\vspace{0.2cm}

 \noindent\textit{Proof of Proposition \ref{prop--diameter bound}.} 
 Recall that we have uniform high regularity estimate for $\omega_t$ away from $p$ and each $\omega_t$ is conical at $p$.
 	Fix a small neighborhood $V$ of $p$. Without lose of generality, we may assume $\operatorname{dist}(x,\partial U)\geq 1$ for any $x\in V$.
 	
 	\noindent \textbf{Step 1.} Let $\rho:\mathbb R\rightarrow [0,\infty)$ be a cut-off function with $\rho=1$ on $(-\infty, 1/2)$ and $\rho=0$ on $[1,\infty)$. For any $x\in V\setminus \{p\}$ and $r\in (0,1)$. Let $u(y)= \rho(d_{\omega}(x,y)/r)$. We are going to show 	
 	\begin{equation}\label{eq--want sobolev}
 		\left( \int_{B_\omega(x, r)}|u|^{2 q} \omega^n\right)^{1 / q} \leq C_2 \int_{B_\omega(x, r)}|\nabla u|_\omega^2 \omega^n.
 	\end{equation}where $q$ and $C_2$ are the constants in Proposition \ref{prop--sobolev inequality}. 
 	If $p\notin \olsi B_{\omega}(x,r)$ then this follows from Proposition \ref{prop--sobolev inequality} directly as $B_{\omega}(x,r)\in U_\delta$ for some $\delta$ sufficiently small. In general, we can use a cut-off argument as follows. Since $\omega_t$ is conical at $p$, for every $\epsilon>0$ sufficiently small and compact set $F\subset U\setminus \{p\}$, there exist cut-off function  $\eta_{\epsilon}:U\rightarrow [0,1]$ such that $\eta_\epsilon\equiv 1$ on $F$, $\eta_{\epsilon}=0$ in a neighborhood of $p$ and moreover 
 	\begin{equation}\label{eq--estimate for cut-off function}
 		\int_U|\nabla \eta_\epsilon|_\omega^2 \omega^n<\epsilon.
 	\end{equation} Since $\eta_{\epsilon}u$ has support on $U\setminus \{p\}$, we obtain 
 	\begin{equation*}
 		\left(\int_{B_\omega(x,r)}\left|\eta_\epsilon u\right|^{2 q} \omega^n\right)^{\frac{1}{q}} \leq C_2\int_{B_\omega(x,r)}\left|\nabla\left(\eta_\epsilon u\right)\right|_\omega^2 \omega^n.
 	\end{equation*} Since both $u$ and $\|\nabla u\|$ are bounded, letting $\epsilon\rightarrow0$ and using \eqref{eq--estimate for cut-off function}, we get \eqref{eq--want sobolev}.
 	
 	\noindent \textbf{Step 2} By \eqref{eq--want sobolev}, we get for $x\in V\setminus \{p\}$ and $r\in (0,1)$,
 \begin{equation}
 	\left(\frac{\operatorname{Vol}_\omega\left(B_\omega(x, r / 2)\right)}{(r / 2)^{\frac{2 q}{q-1}}}\right)^{1 / q} \leq C\frac{\operatorname{Vol}_\omega\left(B_\omega(x, r)\right)}{r^{\frac{2 q}{q-1}}}.
 \end{equation}Applying this to the radii $r_m=2^{-m}r$ and iterating the estimates, one obtains that there exists a constant $\kappa>0$ such that
 \begin{equation}\label{eq--weak non-collapsing}
 	\operatorname{Vol}_\omega\left(B_\omega(x, r)\right)\geq \kappa  r^{\frac{2 q}{q-1}}
 \end{equation} Approximating $p$ by a sequence $x_i\in V\setminus \{p\}$, one can show the same estimate holds for balls centered at $p$.
 
 \noindent\textbf{Step 3} We can argue by contradiction to show that $\operatorname{diam}(U,\omega_t)$ has a uniform upper bound. Suppose not, since $\omega$ has uniform estimate on $U\setminus V$.  we can only find a geodesic contained in $V$ whose length can be arbitrary large. Then this would contradict with \eqref{eq--weak non-collapsing} as we have uniform volume upper bound for $\omega_t$.
 \qed

\bibliographystyle{plain}
\bibliography{ref.bib}

\end{document}